\documentclass[a4paper,oneside,10pt]{article}%
\usepackage{amsmath}
\usepackage{amsfonts}
\usepackage{amssymb}
\usepackage{graphicx}
\usepackage[square,numbers,sort&compress]{natbib}%
\providecommand{\U}[1]{\protect\rule{.1in}{.1in}}

\newtheorem{theorem}{Theorem}[section]

\newtheorem{definition}[theorem]{Definition}

\newtheorem{example}[theorem]{Example}

\newtheorem{lemma}[theorem]{Lemma}

\newtheorem{proposition}[theorem]{Proposition}
\newtheorem{remark}[theorem]{Remark}

\newenvironment{proof}[1][Proof]{\noindent \textbf{#1.} }{\  \rule{0.5em}{0.5em}}
\numberwithin{equation}{section}

\begin{document}

\title{Stochastic maximum principle for stochastic recursive optimal control problem
under volatility ambiguity}
\author{Mingshang Hu \thanks{Qilu Institute of Finance, Shandong University, Jinan,
Shandong 250100, PR China. humingshang@sdu.edu.cn. Research supported by NSF
(No. 11201262 and 10921101) and Shandong Province (No.BS2013SF020 and
ZR2014AP005) }
\and Shaolin Ji\thanks{Qilu Institute of Finance, Shandong University, Jinan,
Shandong 250100, PR China. jsl@sdu.edu.cn (Corresponding author). Research
supported by NSF (No. 11171187, 11222110 and 11221061), Programme of
Introducing Talents of Discipline to Universities of China (No.B12023). Hu and
Ji's research was partially supported by NSF (No. 10921101) and by the 111
Project (No. B12023) }}
\maketitle

\textbf{Abstract}. We study a stochastic recursive optimal control problem in
which the cost functional is described by the solution of a backward
stochastic differential equation driven by $G$-Brownian motion. Some of the
economic and financial optimization problems with volatility ambiguity can be
formulated as such problems. Different from the classical variational
approach, we establish the maximum principle by the linearization and weak
convergence methods.



\addcontentsline{toc}{section}{\hspace*{1.8em}Abstract}

\section{Introduction}

In economic theory, for a given consumption process $(c_{s})_{0\leq s\leq T}$
under probability $P$, Duffie and Epstein \cite{DE} introduced the stochastic
differential recursive utility%
\begin{equation}
y^{P}(t)=E_{P}[\int_{t}^{T}f(y^{P}(s),c(s))ds\mid\mathcal{F}_{t}],\;0\leq
t\leq T, \label{sdu-1}%
\end{equation}
and many optimization problems for the stochastic differential recursive
utilities are well studied by Duffie and Skiadas \cite{DS}\ etc. In fact, the
stochastic differential recursive utility is associated with the solution of a
particular backward stochastic differential equation (BSDE). It is well known
that the general BSDE was introduced by Pardoux and Peng \cite{PP90}. Peng
\cite{peng-1993} first generalized the classical stochastic optimal control
problem to a new one in which the objective functional is defined by the
solution of the following BSDE (\ref{intro--bsde0}) at time $0$:%
\begin{equation}
\left\{
\begin{array}
[c]{rl}%
-dy(t)= & f(t,x(t),y(t),z(t),u(t))dt-z(t)dB(t),\\
y(T)= & \phi(x(T)),
\end{array}
\right.  \label{intro--bsde0}%
\end{equation}
where $B$ is a standard Brownian motion defined on a probability space
$(\Omega,\mathcal{F},P)$. From the BSDE point of view, El Karoui et. al.
\cite{EPQ} considered a more general class of recursive utilities defined as
the solution of BSDEs. Thus, this new kind of stochastic optimal control
problem is called the stochastic recursive optimal control problem.

Chen and Epstein \cite{CE} studied the stochastic differential recursive
utility with drift ambiguity. The drift ambiguity in their context is
described by a class of equivalent probability measures $\mathcal{P}$. The
stochastic differential recursive utility with drift ambiguity is defined as
the lower envelope
\begin{equation}
y(t)=\underset{P\in\mathcal{P}}{\text{ess}\inf}y^{P}(t),\;0\leq t\leq T,
\label{sdu-2}%
\end{equation}
where $y^{P}(t)$ is the solution of (\ref{sdu-1}) at time $t$. They proved
that $y(t)$ of (\ref{sdu-2}) can be characterized by a special BSDE and the
corresponding recursive utility optimization problems with drift ambiguity
still fall in the framework of the stochastic recursive optimal control problem.

Many economic and financial problems involve volatility ambiguity (for the
motivation to consider volatility uncertainty, refer to Epstein and Ji
\cite{EJ-1,EJ-2}). It is well known that volatility ambiguity is chracterized
by a family of nondominated probability measures. In this case, (\ref{sdu-2})
can not be formulated as a classical BSDE, because it can not be modeled
within a probability space framework. So we need a new framework to
accommodate stochastic differential recursive utility with volatility ambiguity.\ 

Inspired by studying financial problems with volatility ambiguity (see
\cite{ALP,L}), Peng introduced a fully nonlinear expectation, called
$G$-expectation $\mathbb{\hat{E}}\mathcal{[\cdot]}$ (see \cite{P10} and the
references therein) which does not require a probability space framework.
Under this $G$-expectation framework ($G$-framework for short) a new type of
Brownian motion called $G$-Brownian motion was constructed. The stochastic
calculus with respect to the $G$-Brownian motion has been established.

Recently, Hu et. al developed the BSDE theory under this $G$-expectation
framework in \cite{HJPS,HJPS1} (see Soner et al. \cite{STZ11} for another
formulation of fully nonlinear BSDE, called 2BSDE). In more details, they
proved that the following BSDE driven by $G$-Brownian motion ($G$-BSDE for
short)%
\[%
\begin{array}
[c]{rl}%
y(t)= & \xi+\int_{t}^{T}f(s,y(s),z(s))ds+\int_{t}^{T}g(s,y(s),z(s))d\langle
B\rangle(s)\\
& -\int_{t}^{T}z(s)dB(s)-(K(T)-K(t))
\end{array}
\]
has a unique triple of solution $(y,z,K)$. In fact, in the volatility
ambiguity case, (\ref{sdu-2}) can be formulated as a special $G$-BSDE (see
\cite{EJ-1,EJ-2}). So the stochastic recursive utility optimization problem
with volatility ambiguity is a special case of the following problem
(\ref{intro-bsde}). The state equations are the following forward and backward
SDEs driven by $G$-Brownian motion: for $t\in\lbrack0,T],$%
\[
\left\{
\begin{array}
[c]{rl}%
dx(t)= & b(t,x(t),u(t))dt+h^{ij}(t,x(t),u(t))d\langle B^{i},B^{j}%
\rangle(t)+\sigma^{i}(t,x(t),u(t))dB^{i}(t),\\
x(0)= & x_{0}\in\mathbb{R}^{n},
\end{array}
\right.
\]%
\begin{equation}
\left\{
\begin{array}
[c]{rl}%
-dy(t)= & f(t,x(t),y(t),z(t),u(t))dt+g^{ij}(t,x(t),y(t),z(t),u(t))d\langle
B^{i},B^{j}\rangle(t)-z(t)dB(t)-dK(t),\\
y(T)= & \phi(x(T)).
\end{array}
\right.  \label{intro-bsde}%
\end{equation}
The cost functional is introduced by the solution of the above BSDE at time
$0$, i.e., $J(u(\cdot))=y(0)$. The stochastic recursive optimal control
problem is to minimize the cost functional over the admissible controls.

The stochastic maximum principle is an important approach to solve stochastic
optimal control problems (see \cite{FHT,Hu,HP,LZ,MOZ,P-90,QT,T,Wu,Yong,Zhou}).
A local form of the stochastic maximum principle for the classical stochastic
recursive optimal control problem was first established in Peng
\cite{peng-1993}. In this paper, we study the stochastic maximum principle for
the problem (\ref{intro-bsde}) when the control domain is convex.

Note that the solution $y$ of (\ref{intro-bsde}) at time $0$ can be written
as
\begin{align}
y_{0}  &  =\mathbb{\hat{E}}[\phi(x(T))+\int_{0}^{T}%
f(t,x(t),y(t),z(t),u(t))dt+\int_{0}^{T}g^{ij}(t,x(t),y(t),z(t),u(t))d\langle
B^{i},B^{j}\rangle(t)]\nonumber\\
&  =\sup_{P\in\mathcal{P}}E_{P}[\phi(x(T))+\int_{0}^{T}%
f(t,x(t),y(t),z(t),u(t))dt+\int_{0}^{T}g^{ij}(t,x(t),y(t),z(t),u(t))d\langle
B^{i},B^{j}\rangle(t)], \label{state-intro-1}%
\end{align}
where $\mathcal{P}$ is a family of weakly compact nondominated probability
measures (see \cite{DHP11}). Thus, our stochastic recursive optimal control
problem is essentially a "inf sup problem". Such problem is known as the
robust optimal control problem, i.e., we consider the worst scenario by
maximizing over a set of probability measures and then we minimize the cost functional.

For the case $f$ does not depend on $(y,z)$ and $g^{ij}=0$, i.e.,%
\begin{equation}
J(u(\cdot))=\mathbb{\hat{E}}[\phi(x(T))+\int_{0}^{T}f(t,x(t),u(t))dt],
\label{classical-objective00}%
\end{equation}
Xu \cite{Xu} studied this problem. Based on the subadditivity of
$\mathbb{\hat{E}}[\cdot]$, he obtained the variational inequality by the
classical variational method. But he did not\ get the stochastic maximum
principle since the sublinear operator $\mathbb{\hat{E}}$ in his main theorem
can not be deleted. It is worth to pointing out that the classical variational
method can not be applied to obtain the variational inequality for our problem
(\ref{intro-bsde}).

In the literatures, in order to derive the maximum principle for the classical
stochastic recursive optimal control problem, one need to obtain the
variational equation for the BSDE (\ref{intro--bsde0}). But in our context,
since the $K$ term of the solution of (\ref{intro-bsde}) is a decreasing
$G$-martingale, it is unable to obtain the "derivative" for $K$ in general. So
we can not obtain the variational equation for the $G$-BSDE (\ref{intro-bsde}%
). To overcome this difficulty, we introduce the linearization and weak
convergence methods to directly obtain the derivative for the value function.
By Minimax Theorem, the variational inequality on a reference probability
$P^{\ast}$ is obtained. Based on the obtained variational inequality, we
derive the stochastic maximum principle holds $P^{\ast}$-a.s.. Furthermore, we
prove that the obtained stochastic maximum principle is also a sufficient
condition under some convex assumptions.

The paper is organized as follows. In Section 2, we present some fundamental
results on $G$-expectation theory. We formulate our stochastic recursive
optimal control problem in Section 3. We derive the maximum principle in
Section 4 and give the general results in Section 5. In Section 6, applying
the obtained maximum principle, we solve a LQ problem.

\section{Preliminaries}

We review some basic notions and results of $G$-expectations. The readers may
refer to \cite{HJPS,P07a,P07b,P08a,P10} for more details.

Let $\Omega=C_{0}([0,\infty);\mathbb{R}^{d})$ be the space of $\mathbb{R}^{d}%
$-valued continuous functions on $[0,\infty)$ with $\omega_{0}=0$ and let
$(B(t))_{t\geq0}$ be the canonical process. For each fixed $T>0$, set%
\[
L_{ip}(\Omega_{T}):=\{ \varphi(B(t_{1}),\cdots,B(t_{n})):n\geq1,t_{1}%
,\cdots,t_{n}\in\lbrack0,T],\varphi\in C_{b.Lip}(\mathbb{R}^{d\times n})\},
\]
where $C_{b.Lip}(\mathbb{R}^{d\times n})$ denotes the space of bounded
Lipschitz functions on $\mathbb{R}^{d\times n}$. Obviously, $L_{ip}(\Omega
_{T})\subset L_{ip}(\Omega_{T^{\prime}})$ for $T<T^{\prime}$. We also set%
\[
L_{ip}(\Omega)=%
{\displaystyle\bigcup\limits_{n=1}^{\infty}}
L_{ip}(\Omega_{n}).
\]

For each given monotonic and sublinear function $G(\cdot):\mathbb{S}%
_{d}\rightarrow\mathbb{R}$, where $\mathbb{S}_{d}$ denotes the collection of
$d\times d$ symmetric matrices, there exists a bounded and closed subset
$\Gamma\subset$$\mathbb{R}$$^{d\times d}$ such that
\begin{equation}
G(A)=\frac{1}{2}\sup_{\gamma\in\Gamma}\text{\textrm{tr}}[\gamma\gamma^{T}A],
\end{equation}
where $\mathbb{R}$$^{d\times d}$ denotes the collection of $d\times d$
matrices. In this paper we only consider non-degenerate $G$, i.e., there
exists some $\underline{\sigma}^{2}>0$ such that $G(A)-G(B)\geq
\underline{\sigma}^{2}\mathrm{tr}[A-B]$ for any $A\geq B$. Now, we define a
functional $\mathbb{\hat{E}}:L_{ip}(\Omega)\rightarrow\mathbb{R}$ by two steps.

Step 1. For $X=\varphi(B(t+s)-B(s))$ with $t$, $s\geq0$ and $\varphi\in
C_{b.Lip}(\mathbb{R}^{d})$, we define%
\[
\mathbb{\hat{E}}[X]=u(t,0),
\]
where $u$ is the solution of the following $G$-heat equation:%
\[
\partial_{t}u-G(D_{xx}^{2}u)=0,\ u(0,x)=\varphi(x).
\]

Step 2. For $X=\varphi(B(t_{1})-B(t_{0}),B(t_{2})-B(t_{1}),\cdots
,B(t_{n})-B(t_{n-1}))$ with $0\leq t_{0}<\cdots<t_{n}$ and $\varphi\in
C_{b.Lip}(\mathbb{R}^{d\times n})$, we define%
\[
\mathbb{\hat{E}}[X]=\varphi_{n},
\]
where $\varphi_{n}$ is obtained via the following procedure:%
\[%
\begin{array}
[c]{rcl}%
\varphi_{1}(x_{1},\cdots,x_{n-1}) & = & \mathbb{\hat{E}}[\varphi(x_{1}%
,\cdots,x_{n-1},B(t_{n})-B(t_{n-1}))],\\
\varphi_{2}(x_{1},\cdots,x_{n-2}) & = & \mathbb{\hat{E}}[\varphi_{1}%
(x_{1},\cdots,x_{n-2},B(t_{n-1})-B(t_{n-2}))],\\
& \vdots & \\
\varphi_{n} & = & \mathbb{\hat{E}}[\varphi_{n-1}(B(t_{1})-B(t_{0}))].
\end{array}
\]
The corresponding conditional expectation $\mathbb{\hat{E}}_{t}$ of $X$ with
$t=t_{i}$ is defined by%
\[%
\begin{array}
[c]{l}%
\mathbb{\hat{E}}_{t_{i}}[\varphi(B(t_{1})-B(t_{0}),B(t_{2})-B(t_{1}%
),\cdots,B(t_{n})-B(t_{n-1}))]\\
=\varphi_{n-i}(B(t_{1})-B(t_{0}),\cdots,B(t_{i})-B(t_{i-1})).
\end{array}
\]

It is easy to check that $(\mathbb{\hat{E}}_{t})_{t\geq0}$ satisfies the
following properties: for each $X$, $Y\in L_{ip}(\Omega)$,

\begin{description}
\item[(i)] Monotonicity: If $X\geq Y$, then $\mathbb{\hat{E}}_{t}%
[X]\geq\mathbb{\hat{E}}_{t}[Y]$;

\item[(ii)] Constant preservation: $\mathbb{\hat{E}}_{t}[X]=X$ for $X\in
L_{ip}(\Omega_{t})$;

\item[(iii)] Sub-additivity: $\mathbb{\hat{E}}_{t}[X+Y]\leq\mathbb{\hat{E}%
}_{t}[X]+\mathbb{\hat{E}}_{t}[Y]$;

\item[(iv)] Positive homogeneity: $\mathbb{\hat{E}}_{t}[XY]=X^{+}%
\mathbb{\hat{E}}_{t}[Y]+X^{-}\mathbb{\hat{E}}_{t}[-Y]$ for $X\in L_{ip}%
(\Omega_{t})$;

\item[(v)] Consistency: $\mathbb{\hat{E}}_{s}[\mathbb{\hat{E}}_{t}%
[X]]=\mathbb{\hat{E}}_{s\wedge t}[X]$, specially, $\mathbb{\hat{E}%
}[\mathbb{\hat{E}}_{t}[X]]=\mathbb{\hat{E}}[X]$.
\end{description}

We denote by $L_{G}^{p}(\Omega)$ the completion of $L_{ip}(\Omega)$ under the
norm $\Vert X\Vert_{p,G}=(\mathbb{\hat{E}}[|X|^{p}])^{1/p}$ for $p\geq1$,
similarly for $L_{G}^{p}(\Omega_{T})$. For each$\ t\geq0$, $\mathbb{\hat{E}%
}_{t}[\cdot]$ can be extended continuously to $L_{G}^{1}(\Omega)$ under the
norm $\Vert\cdot\Vert_{1,G}$. $(\Omega,L_{G}^{1}(\Omega),\mathbb{\hat{E}})$ is
called a $G$-expectation space. The corresponding canonical process
$(B(t))_{t\geq0}$ is called a $G$-Brownian motion.

\begin{definition}
A process $(X(t))_{t\geq0}$ is called a $G$-martingale if $X(t)\in L_{G}%
^{1}(\Omega_{t})$ and $\mathbb{\hat{E}}_{s}[X(t)]=X(s)$ for $s\leq t$.
\end{definition}

\begin{remark}
It is important to note that $(-X(t))_{t\geq0}$ may be not a $G$-martingale.
\end{remark}

Set%
\begin{equation}
\mathcal{P}=\{P:P\text{ is a probability on }(\Omega,\mathcal{B}%
(\Omega))\text{, }E_{P}[X]\leq\mathbb{\hat{E}}[X]\text{ for }X\in L_{G}%
^{1}(\Omega)\}. \label{pr12345}%
\end{equation}

\begin{theorem}
\label{the2.7} (\cite{DHP11,HP09}) Let $\mathcal{P}$ be defined as in
(\ref{pr12345}). Then $\mathcal{P}$ is convex, weakly compact and
\[
\mathbb{\hat{E}}[\xi]=\max_{P\in\mathcal{P}}E_{P}[\xi]\ \ \text{for
\ all}\ \xi\in L_{G}^{1}(\Omega).
\]
$\mathcal{P}$ is called a set that represents $\mathbb{\hat{E}}$.
\end{theorem}

The following proposition is important in our paper.

\begin{proposition}
\label{npro-2.8} (\cite{DHP11}) Let $\{P_{n}:n\geq1\} \subset\mathcal{P}$
converge weakly to $P$. Then for each $\xi\in L_{G}^{1}(\Omega)$, we have
$E_{P_{n}}[\xi]\rightarrow E_{P}[\xi]$.
\end{proposition}

\begin{definition}
\label{def2.6} Let $M_{G}^{0}(0,T)$ be the collection of processes in the
following form: for a given partition $\{t_{0},\cdot\cdot\cdot,t_{N}\}=\pi
_{T}$ of $[0,T]$,
\[
\eta(t)=\sum_{j=0}^{N-1}\xi_{j}I_{[t_{j},t_{j+1})}(t),
\]
where $\xi_{i}\in L_{ip}(\Omega_{t_{i}})$, $i=0,1,2,\cdot\cdot\cdot,N-1$.
\end{definition}

We denote by $M_{G}^{p}(0,T)$ the completion of $M_{G}^{0}(0,T)$ under the
norm $\Vert\eta\Vert_{M_{G}^{p}}=\{ \mathbb{\hat{E}}[\int_{0}^{T}|\eta
(s)|^{p}ds]\}^{1/p}$ for $p\geq1$. The It\^{o}'s integral $\int_{0}^{T}%
\eta(s)dB(s)$ is well defined for $\eta\in M_{G}^{2}(0,T)$.

\section{Stochastic optimal control problem}

We first give the definition of admissible controls.

\begin{definition}
$u(\cdot)$ is said to be an admissible control on $[0,T]$, if it satisfies the
following conditions:

\begin{description}
\item[(i)] $u(\cdot):[0,T]\times\Omega\rightarrow U$ where $U$ is a nonempty
convex subset of $\mathbb{R}^{m}$;

\item[(ii)] $u(\cdot)\in M_{G}^{\beta}(0,T;\mathbb{R}^{m})$ with $\beta>2$.
\end{description}
\end{definition}

The set of admissible controls is denoted by $\mathcal{U}[0,T]$.

In the rest of this paper, we use the Einstein summation convention.

Let $u(\cdot)\in\mathcal{U}[0,T]$. Consider the following forward and backward
SDEs driven by $G$-Brownian motion: for $t\in\lbrack0,T],$%
\begin{equation}
\left\{
\begin{array}
[c]{rl}%
dx(t)= & b(t,x(t),u(t))dt+h^{ij}(t,x(t),u(t))d\langle B^{i},B^{j}%
\rangle(t)+\sigma^{i}(t,x(t),u(t))dB^{i}(t),\\
x(0)= & x_{0}\in\mathbb{R}^{n},
\end{array}
\right.  \label{state-1}%
\end{equation}%
\begin{equation}
\left\{
\begin{array}
[c]{rl}%
-dy(t)= & f(t,x(t),y(t),z(t),u(t))dt+g^{ij}(t,x(t),y(t),z(t),u(t))d\langle
B^{i},B^{j}\rangle(t)-z(t)dB(t)-dK(t),\\
y(T)= & \phi(x(T)),
\end{array}
\right.  \label{state-2}%
\end{equation}
where
\[%
\begin{array}
[c]{l}%
b:[0,T]\times\mathbb{R}^{n}\times U\rightarrow\mathbb{R}^{n}\text{;}\\
h^{ij}:[0,T]\times\mathbb{R}^{n}\times U\rightarrow\mathbb{R}^{n}\text{;}\\
\sigma=[\sigma^{1},\ldots,\sigma^{d}]:[0,T]\times\mathbb{R}^{n}\times
U\rightarrow\mathbb{R}^{n\times d}\text{;}\\
f:[0,T]\times\mathbb{R}^{n}\times\mathbb{R}\times\mathbb{R}^{1\times d}\times
U\rightarrow\mathbb{R}\text{;}\\
g^{ij}:[0,T]\times\mathbb{R}^{n}\times\mathbb{R}\times\mathbb{R}^{1\times
d}\times U\rightarrow\mathbb{R}\text{;}\\
\phi:\mathbb{R}^{n}\rightarrow\mathbb{R}.
\end{array}
\]

Denote%
\[%
\begin{array}
[c]{l}%
S_{G}^{0}(0,T)=\{h(t,B_{t_{1}\wedge t},\cdot\cdot\cdot,B_{t_{n}\wedge
t}):t_{1},\ldots,t_{n}\in\lbrack0,T],h\in C_{b,Lip}(\mathbb{R}^{n+1})\};\\
S_{G}^{2}(0,T)=\{ \text{the completion of }S_{G}^{0}(0,T)\text{ under the norm
}\Vert\eta\Vert_{S_{G}^{2}}=\{ \mathbb{\hat{E}}[\sup_{t\in\lbrack0,T]}%
|\eta_{t}|^{2}]\}^{\frac{1}{2}}\}.
\end{array}
\]

For given $u(\cdot)\in\mathcal{U}[0,T]$, $x(\cdot)$ and $(y(\cdot
),z(\cdot),K(\cdot))$ are called solutions of the above forward and backward
SDEs respectively if $x(\cdot)\in M_{G}^{2}(0,T;\mathbb{R}^{n})$;
$(y(\cdot),z(\cdot))\in S_{G}^{2}(0,T)\times M_{G}^{2}(0,T;\mathbb{R}^{1\times
d})$; $K(\cdot)$ is a decreasing $G$-martingale$\ $with $K(0)=0$ and $K(T)\in
L_{G}^{2}(\Omega_{T});$ (\ref{state-1}) and (\ref{state-2}) are satisfied respectively.

We assume:\medskip

\begin{description}
\item[(H1)] $b,h^{ij},\sigma,f,g^{ij},\phi$ are continuous and differentiable
in $(x,y,z,u);$

\item[(H2)] The derivatives of $b,h^{ij},\sigma,f,g^{ij},\phi$ in $(x,y,z,u)$
are bounded;

\item[(H3)] There exists a modulus of continuity $\bar{\omega}:[0,\infty
)\rightarrow\lbrack0,\infty)$ such that for any $t\in\lbrack0,T]$, $x$,
$x^{\prime}\in\mathbb{R}^{n}$, $y$, $y^{\prime}\in\mathbb{R}$, $z$,
$z^{\prime}\in\mathbb{R}^{1\times d}$, $u$, $u^{\prime}\in\mathbb{R}^{m}$,%
\[
|\varphi(t,x,y,z,u)-\varphi(t,x^{\prime},y^{\prime},z^{\prime},u^{\prime
})|\leq\bar{\omega}(|x-x^{\prime}|+|y-y^{\prime}|+|z-z^{\prime}|+|u-u^{\prime
}|),
\]
where $\varphi$ is the derivatives of $b,h^{ij},\sigma,f,g^{ij},\phi$ in
$(x,y,z,u)$.
\end{description}

We have the following theorems.

\begin{theorem}
(\cite{P10}) Let assumptions (H1)-(H2) hold. Then (\ref{state-1}) has a unique
solution $x(\cdot)$.
\end{theorem}

\begin{theorem}
(\cite{HJPS1}) Let assumptions (H1)-(H2) hold. Then (\ref{state-2}) has a
unique solution $(y(\cdot),z(\cdot),K(\cdot))$.
\end{theorem}

The state equation of our stochastic optimal control problem is governed by
the above forward and backward SDEs (\ref{state-1}) and (\ref{state-2}). The
cost functional is introduced by the solution of the BSDE (\ref{state-2}) at
time $0$, i.e.,
\[
J(u(\cdot))=y(0).
\]
The stochastic optimal control problem is to minimize the cost functional over
$\mathcal{U}[0,T]$.

\begin{remark}
We point out that $\mathcal{U}[0,T]$ contains all feedback controls (see Hu
and Ji \cite{HJ}). In the last section, we show that the optimal control of
the LQ problem is a special kind of feedback control.
\end{remark}

In summary, our stochastic control problem is%
\[
\left\{
\begin{array}
[c]{rl}%
\text{Minimize} & J(u(\cdot))\\
\text{subject to} & u(\cdot)\in\mathcal{U}[0,T].
\end{array}
\right.
\]

\section{Stochastic Maximum Principle}

In this section, to ease the presentation we only study the case where
$h_{ij}\equiv0$, $g_{ij}\equiv0$ and $f$ does not include $z$ term. We will
present the results for the general case in Section 5.

\subsection{Variational equation}

Let $\bar{u}(\cdot)$ be optimal and $(\bar{x}(\cdot),\bar{y}(\cdot),\bar
{z}(\cdot),\bar{K}(\cdot))$ be the corresponding state processes of
(\ref{state-1}) and (\ref{state-2}). Take an arbitrary $u(\cdot)\in
\mathcal{U}[0,T]$. Since $\mathcal{U}[0,T]$ is convex, then, for each
$0\leq\rho\leq1$, $\bar{u}(\cdot)+\rho(u(\cdot)-\bar{u}(\cdot))\in
\mathcal{U}[0,T]$. Let $(x_{\rho}(\cdot),y_{\rho}(\cdot),z_{\rho}%
(\cdot),K_{\rho}(\cdot))$ be the state processes of (\ref{state-1}) and
(\ref{state-2}) associated with $\bar{u}(\cdot)+\rho(u(\cdot)-\bar{u}(\cdot))$.

To derive the first-order necessary condition in terms of small $\rho$, let
$\hat{x}(\cdot)$ be the solution of the following SDE:%
\begin{equation}
\left\{
\begin{array}
[c]{rl}%
d\hat{x}(t)= & [b_{x}(t)\hat{x}(t)+b_{u}(t)(u(t)-\bar{u}(t))]dt+[\sigma
_{x}^{i}(t)\hat{x}(t)+\sigma_{u}^{i}(t)(u(t)-\bar{u}(t))]dB^{i}(t),\\
\hat{x}(0)= & 0,
\end{array}
\right.  \label{variational-eq-1}%
\end{equation}
where $b_{x}(t)=b_{x}(t,\bar{x}(t),\bar{u}(t))$, $b_{u}(t)=b_{u}(t,\bar
{x}(t),\bar{u}(t))$, $\sigma_{x}^{i}(t)=\sigma_{x}^{i}(t,\bar{x}(t),\bar
{u}(t))$, $\sigma_{u}^{i}(t)=\sigma_{u}^{i}(t,\bar{x}(t),\bar{u}(t))$.

In this paper, we define%
\[
b_{x}(t)=\left[
\begin{array}
[c]{ccc}%
b_{1x_{1}}(t), & \cdots & ,b_{1x_{n}}(t)\\
\vdots &  & \vdots\\
b_{nx_{1}}(t), & \cdots & ,b_{nx_{n}}(t)
\end{array}
\right]  .
\]
The other derivatives are defined similarly.

Equation (\ref{variational-eq-1}) is called the variational equation for SDE
(\ref{state-1}). By Theorem 1.2 in \cite{P10}, there exists a unique solution
$\hat{x}(\cdot)\in M_{G}^{2}(0,T;\mathbb{R}^{n})$ to equation
(\ref{variational-eq-1}).

Set%
\[
\tilde{x}_{\rho}(t)=\rho^{-1}[x_{\rho}(t)-\bar{x}(t)]-\hat{x}(t).
\]

\begin{proposition}
\label{variational derivative-forward}Assume (H1)-(H3) hold. Then

\begin{description}
\item[(i)] there exists a positive constant $C$ such that $\mathbb{\hat{E}%
}[\mid\tilde{x}_{\rho}(t)\mid^{2}]\leq C$ for $0\leq\rho\leq1$;

\item[(ii)] $\underset{\rho\rightarrow0}{\lim}\underset{0\leq t\leq T}{\sup
}\mathbb{\hat{E}}[\mid\tilde{x}_{\rho}(t)\mid^{2}]=0.$
\end{description}
\end{proposition}

In the following, we always use the constant $C$ for simplicity, where $C$ can
be change from line to line. For prove this proposition, we need the following lemma.

\begin{lemma}
\label{uniform property} Suppose that $\eta$ belongs to $M_{G}^{1}(0,T)$. Then
for each $\varepsilon>0$, there exists a positive number $\delta$ such that
$\mathbb{\hat{E}}[\int_{0}^{T}\mid\eta\mid I_{A}dt]<\varepsilon$ for any
$A\in\mathcal{B}([0,T])\times\mathcal{F}_{T}$ with $\mathbb{\hat{E}}[\int%
_{0}^{T}I_{A}(t,\omega)dt]<\delta$.
\end{lemma}

\begin{proof}
Since $\eta\in M_{G}^{1}(0,T)$, we have%
\[
\underset{N\rightarrow\infty}{\lim}\mathbb{\hat{E}}[\int_{0}^{T}\mid\eta\mid
I_{\{ \mid\eta\mid\geq N\}}dt]=0
\]
by Proposition 18 in \cite{DHP11}. Then for any $\varepsilon>0$, there exists
a $N_{0}$ such that $\mathbb{\hat{E}}[\int_{0}^{T}\mid\eta\mid I_{\{ \mid
\eta\mid\geq N_{0}\}}dt]<\frac{\varepsilon}{2}$. Take $\delta=\frac
{\varepsilon}{2N_{0}}$. For any $A\in\mathcal{B}([0,T])\times\mathcal{F}_{T}$
with $\mathbb{\hat{E}}[\int_{0}^{T}I_{A}(t,\omega)dt]<\delta$, we have that%
\[%
\begin{array}
[c]{cl}%
\mathbb{\hat{E}}[\int_{0}^{T}\mid\eta\mid I_{A}dt] & =\mathbb{\hat{E}}%
[\int_{0}^{T}\mid\eta\mid(I_{\{ \mid\eta\mid\geq N_{0}\} \cap A}+I_{\{
\mid\eta\mid<N_{0}\} \cap A})dt]\\
& \leq\mathbb{\hat{E}}[\int_{0}^{T}\mid\eta\mid I_{\{ \mid\eta\mid\geq N_{0}\}
\cap A}dt]+\mathbb{\hat{E}}[\int_{0}^{T}\mid\eta\mid I_{\{ \mid\eta\mid
<N_{0}\} \cap A}dt]\\
& \leq\mathbb{\hat{E}}[\int_{0}^{T}\mid\eta\mid I_{\{ \mid\eta\mid\geq
N_{0}\}}dt]+\mathbb{\hat{E}}[\int_{0}^{T}N_{0}I_{\{ \mid\eta\mid<N_{0}\} \cap
A}dt]\\
& \leq\varepsilon.
\end{array}
\]
This completes the proof.
\end{proof}

\textbf{Proof of Proposition \ref{variational derivative-forward}. (i)} From
(\ref{state-1}) and (\ref{variational-eq-1}), we have%
\[
\left\{
\begin{array}
[c]{rl}%
d\tilde{x}_{\rho}(t)= & \rho^{-1}[b_{\rho}(t)-b(t)-\rho(b_{x}(t)\hat
{x}(t)+b_{u}(t)(u(t)-\bar{u}(t)))]dt\\
& +\rho^{-1}[\sigma_{\rho}^{i}(t)-\sigma^{i}(t)-\rho(\sigma_{x}^{i}(t)\hat
{x}(t)+\sigma_{u}^{i}(t)(u(t)-\bar{u}(t)))]dB^{i}(t),\\
\tilde{x}_{\rho}(0)= & 0.
\end{array}
\right.
\]
where $b_{\rho}(t)=b(t,x_{\rho}(t),\bar{u}(t)+\rho(u(t)-\bar{u}(t)))$,
$b(t)=b(t,\bar{x}(t),\bar{u}(t)),$ $\sigma_{\rho}^{i}(t)=\sigma^{i}(t,x_{\rho
}(t),\bar{u}(t)+\rho(u(t)-\bar{u}(t)))$and $\sigma^{i}(t)=\sigma^{i}(t,\bar
{x}(t),\bar{u}(t))$. Let%
\begin{align*}
A_{\rho}(t)  &  =\int_{0}^{1}b_{x}(t,\bar{x}(t)+\lambda\rho(\hat{x}%
(t)+\tilde{x}_{\rho}(t)),\bar{u}(t)+\lambda\rho(u(t)-\bar{u}(t)))d\lambda,\\
B_{\rho}^{i}(t)  &  =\int_{0}^{1}\sigma_{x}^{i}(t,\bar{x}(t)+\lambda\rho
(\hat{x}(t)+\tilde{x}_{\rho}(t)),\bar{u}(t)+\lambda\rho(u(t)-\bar
{u}(t)))d\lambda,\\
C_{\rho}(t)  &  =[A_{\rho}(t)-b_{x}(t)]\hat{x}(t)+\int_{0}^{1}[b_{u}(t,\bar
{x}(t)+\lambda\rho(\hat{x}(t)+\tilde{x}_{\rho}(t)),\bar{u}(t)+\lambda
\rho(u(t)-\bar{u}(t)))-b_{u}(t)](u(t)-\bar{u}(t))d\lambda,\\
D_{\rho}^{i}(t)  &  =[B_{\rho}(t)-\sigma_{x}(t)]\hat{x}(t)+\int_{0}^{1}%
[\sigma_{u}^{i}(t,\bar{x}(t)+\lambda\rho(\hat{x}(t)+\tilde{x}_{\rho}%
(t)),\bar{u}(t)+\lambda\rho(u(t)-\bar{u}(t)))-\sigma_{u}^{i}(t)](u(t)-\bar
{u}(t))d\lambda.
\end{align*}
Thus,%
\[
\left\{
\begin{array}
[c]{rl}%
d\tilde{x}_{\rho}(t)= & [A_{\rho}(t)\tilde{x}_{\rho}(t)+C_{\rho}%
(t)]dt+[B_{\rho}^{i}(t)\tilde{x}_{\rho}(t)+D_{\rho}^{i}(t)]dB^{i}(t),\\
\tilde{x}_{\rho}(0)= & 0.
\end{array}
\right.
\]
Using It\^{o}'s formula to $\mid\tilde{x}_{\rho}(t)\mid^{2}$, we get%
\begin{align*}
\mathbb{\hat{E}}[  &  \mid\tilde{x}_{\rho}(t)\mid^{2}]\\
&  =\mathbb{\hat{E}}[\int_{0}^{t}2\langle\tilde{x}_{\rho}(s),A_{\rho}%
(s)\tilde{x}_{\rho}(s)+C_{\rho}(s)\rangle ds+\int_{0}^{t}\langle B_{\rho}%
^{i}(s)\tilde{x}_{\rho}(s)+D_{\rho}^{i}(s),B_{\rho}^{j}(s)\tilde{x}_{\rho
}(s)+D_{\rho}^{j}(s)\rangle d\langle B^{i},B^{j}\rangle(s)]\\
&  \leq C(\mathbb{\hat{E}[}\int_{0}^{t}\mid\tilde{x}_{\rho}(s)\mid
^{2}ds]+I_{\rho}),
\end{align*}
where $C$ is a constant and
\[
I_{\rho}=\mathbb{\hat{E}[}\int_{0}^{T}(\mid C_{\rho}(s)\mid^{2}+\mid D_{\rho
}^{i}(s)\mid^{2})ds].
\]
Applying Gronwall's inequality, we obtain that%
\begin{equation}
\mathbb{\hat{E}}[\mid\tilde{x}_{\rho}(t)\mid^{2}]\leq Ce^{Ct}I_{\rho}\leq
Ce^{CT}I_{\rho}. \label{gronwall ineq}%
\end{equation}
Note that $C_{\rho}(t)$ and $D_{\rho}(t)$ are bounded by $C^{\prime}(\mid
\hat{x}(t)\mid+\mid u(t)-\bar{u}(t)\mid)$, where $C^{\prime}$ is a constant
which is independent with $\rho$. Thus, $\mathbb{\hat{E}}[\mid\tilde{x}_{\rho
}(t)\mid^{2}]$ is bounded by some constant $C$ for $0\leq\rho\leq1$.

\textbf{(ii)} By (\ref{gronwall ineq}), we only need to prove that $I_{\rho
}\rightarrow0$ as $\rho\rightarrow0$. We first prove
\[
\underset{\rho\rightarrow0}{\lim}\mathbb{\hat{E}}[\int_{0}^{T}\mid C_{\rho
}(s)\mid^{2}ds]=0.
\]
Define%
\begin{align*}
E_{\rho}(t)  &  =b_{x}(t,\bar{x}(t)+\lambda\rho(\hat{x}(t)+\tilde{x}_{\rho
}(t)),\bar{u}(t)+\lambda\rho(u(t)-\bar{u}(t)))-b_{x}(t),\\
F_{\rho}(t)  &  =b_{u}(t,\bar{x}(t)+\lambda\rho(\hat{x}(t)+\tilde{x}_{\rho
}(t)),\bar{u}(t)+\lambda\rho(u(t)-\bar{u}(t)))-b_{u}(t).
\end{align*}
For $N>0$, set%
\begin{align*}
S_{1,N}  &  =\{ \mid\hat{x}(t)+\tilde{x}_{\rho}(t)\mid\leq N\},\\
S_{2,N}  &  =\{ \mid u(t)-\bar{u}(t)\mid\leq N\}.
\end{align*}
We have%
\begin{equation}%
\begin{array}
[c]{rl}%
\mid C_{\rho}(t)\mid^{2}= & \mid\int_{0}^{1}E_{\rho}(t)d\lambda\hat{x}%
(t)+\int_{0}^{1}F_{\rho}(t)d\lambda(u(t)-\bar{u}(t))\mid^{2}\\
\leq & 2(\int_{0}^{1}\mid E_{\rho}(t)\mid^{2}d\lambda\mid\hat{x}(t)\mid
^{2}+\int_{0}^{1}\mid F_{\rho}(t)\mid^{2}d\lambda\mid u(t)-\bar{u}(t)\mid
^{2})\\
\leq & 2(\int_{0}^{1}\mid E_{\rho}(t)\mid^{2}(I_{S_{1,N}\cap S_{2,N}%
}+I_{S_{1,N}^{c}}+I_{S_{2,N}^{c}})d\lambda\mid\hat{x}(t)\mid^{2}\\
& +\int_{0}^{1}\mid F_{\rho}(t)\mid^{2}(I_{S_{1,N}\cap S_{2,N}}+I_{S_{1,N}%
^{c}}+I_{S_{2,N}^{c}})d\lambda\mid u(t)-\bar{u}(t)\mid^{2})\\
\leq & 2\bar{\omega}(2N\rho)\mid\hat{x}(t)\mid^{2}+C(I_{S_{1,N}^{c}%
}+I_{S_{2,N}^{c}})\mid\hat{x}(t)\mid^{2}\\
& +2\bar{\omega}(2N\rho)\mid u(t)-\bar{u}(t)\mid^{2}+C(I_{S_{1,N}^{c}%
}+I_{S_{2,N}^{c}})\mid u(t)-\bar{u}(t)\mid^{2}.
\end{array}
\label{estimate C}%
\end{equation}
By Lemma \ref{uniform property}, for each $\varepsilon>0$, there exists a
$\delta>0$ such that for any $A\in\mathcal{B}([0,T])\times\mathcal{F}_{T}$
with $\mathbb{\hat{E}}[\int_{0}^{T}I_{A}(t,\omega)dt]<\delta$, we have that%
\begin{align*}
\mathbb{\hat{E}}[\int_{0}^{T}  &  \mid\hat{x}(t)\mid^{2}I_{A}dt]<\varepsilon
,\\
\mathbb{\hat{E}}[\int_{0}^{T}  &  \mid u(t)-\bar{u}(t)\mid^{2}I_{A}%
dt]<\varepsilon.
\end{align*}
Note that%
\[
\mathbb{\hat{E}}[\int_{0}^{T}(I_{S_{1,N}^{c}}+I_{S_{2,N}^{c}})dt]\leq\frac
{1}{N^{2}}\mathbb{\hat{E}}[\int_{0}^{T}(\mid\hat{x}(t)+\tilde{x}_{\rho}%
(t)\mid^{2}+\mid u(t)-\bar{u}(t)\mid^{2})dt],
\]
then we can choose an $N>0$ such that $\mathbb{\hat{E}}[\int_{0}%
^{T}(I_{S_{1,N}^{c}}+I_{S_{2,N}^{c}})dt]<\delta$, which implies that%
\[
\mathbb{\hat{E}}[\int_{0}^{T}(I_{S_{1,N}^{c}}+I_{S_{2,N}^{c}})(\mid\hat
{x}(t)\mid^{2}+\mid u(t)-\bar{u}(t)\mid^{2})dt]\leq C\varepsilon.
\]
Thus by (\ref{estimate C}), it is easy to obtain $\underset{\rho
\rightarrow0}{\lim}\mathbb{\hat{E}}[\int_{0}^{T}\mid C_{\rho}(s)\mid^{2}%
ds]=0$. Similarly, we can prove that $\underset{\rho\rightarrow0}{\lim
}\mathbb{\hat{E}}[\int_{0}^{T}\mid D_{\rho}(s)\mid^{2}ds]=0$. Thus we get
$\underset{\rho\rightarrow0}{\lim}\underset{0\leq t\leq T}{\sup}%
\mathbb{\hat{E}}[\mid\tilde{x}_{\rho}(t)\mid^{2}]=0$. $\Box$

Now let
\begin{align*}
f_{\rho}(t)  &  =f(t,x_{\rho}(t),y_{\rho}(t),\bar{u}(t)+\rho(u(t)-\bar
{u}(t)))\text{, }f(t)=f(t,\bar{x}(t),\bar{y}(t),\bar{u}(t)),\\
f_{x}(t)  &  =f_{x}(t,\bar{x}(t),\bar{y}(t),\bar{u}(t)),\text{ }f_{y}%
(t)=f_{y}(t,\bar{x}(t),\bar{y}(t),\bar{u}(t)),\text{ }f_{u}(t)=f_{u}(t,\bar
{x}(t),\bar{y}(t),\bar{u}(t)).
\end{align*}

Set%
\[
\mathcal{P}^{\ast}=\{P\in\mathcal{P}\mid E_{P}[\bar{K}(T)]=0\}
\]
and%
\[
\Theta^{u}=\phi_{x}(\bar{x}(T))\hat{x}(T)m(T)+\int_{0}^{T}[f_{x}(s)\hat
{x}(s)+f_{u}(s)(u(s)-\bar{u}(s))]m(s)ds,
\]
where%
\[
m(t)=\exp\{ \int_{0}^{t}f_{y}(s)ds\}.
\]

\begin{theorem}
\label{variational derivative-y}Suppose (H1)-(H3) hold. Then, for any
$u(\cdot)\in\mathcal{U}[0,T]$, there exists a $P^{u}\in\mathcal{P}^{\ast}$
such that%
\begin{equation}
\underset{\rho\rightarrow0}{\lim}\frac{y_{\rho}(0)-\bar{y}(0)}{\rho}=E_{P^{u}%
}\mathbb{[}\Theta^{u}]=\sup_{P\in\mathcal{P}^{\ast}}E_{P}[\Theta^{u}].
\label{new-varequ}%
\end{equation}

\end{theorem}

\begin{remark}
If $B$ is the classical Brownian motion, then $E_{P^{u}}\mathbb{[}\Theta^{u}]$
is the solution of the variational equation for BSDE at time $0$.
\end{remark}

In order to prove this theorem, we need the following lemma.

\begin{lemma}
\label{variational derivative-backward} Assume (H1)-(H3) hold. Then we have

\begin{description}
\item[(i)] $\mathbb{\hat{E}}[\mid J_{1}m(T)\mid]=o(\rho),$

\item[(ii)] $\mathbb{\hat{E}}[\mid\int_{0}^{T}f_{x}(s)\tilde{x}_{\rho
}(s)m(s)ds\mid]=o(1),$

\item[(iii)] $\mathbb{\hat{E}}[\mid\int_{0}^{T}J_{2}(s)m(s)ds\mid]=o(\rho),$
\end{description}

where
\begin{align*}
J_{1}  &  =\phi(x_{\rho}(T))-\phi(\bar{x}(T))-\phi_{x}(\bar{x}(T))\rho\hat
{x}(T),\\
J_{2}(s)  &  =(f_{\rho}(s)-f(s))-[f_{x}(s)(x_{\rho}(s)-\bar{x}(s))+f_{y}%
(s)(y_{\rho}(s)-\bar{y}(s))+f_{u}(s)\rho(u(s)-\bar{u}(s))].
\end{align*}

\end{lemma}

\begin{proof}
\textbf{(i)}%
\[%
\begin{array}
[c]{rl}%
J_{1}= & \int\nolimits_{0}^{1}\phi_{x}(\bar{x}(T)+\lambda(x_{\rho}(T)-\bar
{x}(T))d\lambda(x_{\rho}(T)-\bar{x}(T))-\phi_{x}(\bar{x}(T))\rho\hat{x}(T)\\
= & \int\nolimits_{0}^{1}\phi_{x}(\bar{x}(T)+\lambda\rho(\tilde{x}_{\rho
}(T)+\hat{x}(T))d\lambda\rho(\tilde{x}_{\rho}(T)+\hat{x}(T))-\phi_{x}(\bar
{x}(T))\rho\hat{x}(T)\\
= & \rho\int\nolimits_{0}^{1}(\phi_{x}(\bar{x}(T)+\lambda\rho(\tilde{x}_{\rho
}(T)+\hat{x}(T))-\phi_{x}(\bar{x}(T)))d\lambda\hat{x}(T)\\
& +\rho\int\nolimits_{0}^{1}\phi_{x}(\bar{x}(T)+\lambda\rho(\tilde{x}_{\rho
}(T)+\hat{x}(T))d\lambda\tilde{x}_{\rho}(T).
\end{array}
\]
Using the similar analysis as in Proposition
\ref{variational derivative-forward}, we can prove that
\[
\underset{\rho\rightarrow0}{\lim}\mathbb{\hat{E}}[\mid\int\nolimits_{0}%
^{1}(\phi_{x}(\bar{x}(T)+\lambda\rho(\tilde{x}_{\rho}(T)+\hat{x}(T))-\phi
_{x}(\bar{x}(T)))d\lambda\mid\mid\hat{x}(T)m(T)\mid]=0.
\]
It is easy to see
\[
\mathbb{\hat{E}}[\mid\int\nolimits_{0}^{1}\phi_{x}(\bar{x}(T)+\lambda
\rho(\tilde{x}_{\rho}(T)+\hat{x}(T))d\lambda\tilde{x}_{\rho}(T)m(T)\mid]\leq
C(\mathbb{\hat{E}}[\mid\tilde{x}_{\rho}(T)\mid^{2}])^{\frac{1}{2}%
}(\mathbb{\hat{E}}[\mid m(T)\mid^{2}])^{\frac{1}{2}}.
\]
Then, by Proposition \ref{variational derivative-forward},
\[
\underset{\rho\rightarrow0}{\lim}\mathbb{\hat{E}}[\mid\int\nolimits_{0}%
^{1}\phi_{x}(\bar{x}(T)+\lambda\rho(\tilde{x}_{\rho}(T)+\hat{x}(T))d\lambda
\tilde{x}_{\rho}(T)m(T)\mid]=0.
\]

\textbf{(ii)}%
\[%
\begin{array}
[c]{cl}%
\mathbb{\hat{E}}[\mid\int_{0}^{T}f_{x}(s)\tilde{x}_{\rho}(s)m(s)ds\mid] & \leq
C\int_{0}^{T}\mathbb{\hat{E}}[\mid\tilde{x}_{\rho}(s)\mid\mid m(s)\mid]ds\\
& \leq C\int_{0}^{T}(\mathbb{\hat{E}}[\mid\tilde{x}_{\rho}(s)\mid^{2}%
])^{\frac{1}{2}}ds\\
& \leq CT(\underset{0\leq s\leq T}{\sup}\mathbb{\hat{E}}[\mid\tilde{x}_{\rho
}(s)\mid^{2}])^{\frac{1}{2}}.
\end{array}
\]
By Proposition \ref{variational derivative-forward}, $\mathbb{\hat{E}}%
[\mid\int_{0}^{T}f_{x}(s)\tilde{x}_{\rho}(s)m(s)ds\mid]\rightarrow0$ as
$\rho\rightarrow0$.

\textbf{(iii)} Set%
\[
\tilde{f}_{l}(s)=f_{l}(s,\bar{x}(s)+\lambda\rho(\tilde{x}_{\rho}(s)+\hat
{x}(s)),\bar{y}(s)+\lambda(y_{\rho}(s)-\bar{y}(s)),\bar{u}(s)+\lambda
\rho(u(s)-\bar{u}(s)))
\]
for $l=x,y,u$. Then%
\[%
\begin{array}
[c]{rl}%
J_{2}(s)= & \int\nolimits_{0}^{1}(\tilde{f}_{x}(s)-f_{x}(s))d\lambda
\rho(\tilde{x}_{\rho}(s)+\hat{x}(s))+\int\nolimits_{0}^{1}(\tilde{f}%
_{y}(s)-f_{y}(s))d\lambda(y_{\rho}(s)-\bar{y}(s))\\
& +\int\nolimits_{0}^{1}(\tilde{f}_{u}(s)-f_{u}(s))d\lambda\rho(u(s)-\bar
{u}(s)).
\end{array}
\]
We only prove that%
\[
\mathbb{\hat{E}}[\mid\int_{0}^{T}\int\nolimits_{0}^{1}(\tilde{f}_{y}%
(s)-f_{y}(s))d\lambda(y_{\rho}(s)-\bar{y}(s))m(s)ds\mid]=o(\rho).
\]
The proofs of the other terms are similar.

By Proposition 2.15 in \cite{HJPS}, we have%
\[%
\begin{array}
[c]{l}%
\mid y_{\rho}(t)-\bar{y}(t)\mid^{2}\\
\leq C(\mathbb{\hat{E}}_{t}[\mid\phi(x_{\rho}(T))-\phi(\bar{x}(T))\mid^{2}]\\
\ \ +\mathbb{\hat{E}}_{t}[\int_{t}^{T}\mid f(s,x_{\rho}(s),\bar{y}(s),\bar
{u}(s)+\rho(u(s)-\bar{u}(s)))-f(s)\mid^{2}ds]).
\end{array}
\]
Then, by Proposition \ref{variational derivative-forward},%
\[%
\begin{array}
[c]{l}%
\underset{0\leq t\leq T}{\sup}\mathbb{\hat{E}}[\mid y_{\rho}(t)-\bar{y}%
(t)\mid^{2}]\\
\leq C(\mathbb{\hat{E}}[\mid x_{\rho}(T)-\bar{x}(T)\mid^{2}]\\
\ \ +\int_{0}^{T}(\mathbb{\hat{E}}[\mid x_{\rho}(t)-\bar{x}(t)\mid
^{2}]+\mathbb{\hat{E}}[\rho^{2}\mid u(t)-\bar{u}(t)\mid^{2}])dt)\\
\leq C(\mathbb{\hat{E}}[\rho^{2}\mid\tilde{x}_{\rho}(t)+\hat{x}(t)\mid^{2}]\\
\ \ +\int_{0}^{T}(\mathbb{\hat{E}}[\rho^{2}\mid\tilde{x}_{\rho}(t)+\hat
{x}(t)\mid^{2}]+\mathbb{\hat{E}}[\rho^{2}\mid u(t)-\bar{u}(t)\mid^{2}])dt)\\
\leq C\rho^{2}.
\end{array}
\]

Let $\alpha\in(0,1)$ be fixed. For each $N>0$, we have%
\[%
\begin{array}
[c]{l}%
\mathbb{\hat{E}}[|\int_{0}^{T}\int\nolimits_{0}^{1}(\tilde{f}_{y}%
(s)-f_{y}(s))I_{\{ \mid y_{\rho}(s)-\bar{y}(s)\mid>N\rho\}}d\lambda(y_{\rho
}(s)-\bar{y}(s))m(s)ds|]\\
\leq C\mathbb{\hat{E}}[\int_{0}^{T}I_{\{ \mid y_{\rho}(s)-\bar{y}(s)\mid
>N\rho\}}\mid y_{\rho}(s)-\bar{y}(s)\mid m(s)ds]\\
\leq\frac{C}{N^{\alpha}\rho^{\alpha}}\mathbb{\hat{E}}[\int_{0}^{T}\mid
y_{\rho}(s)-\bar{y}(s)\mid^{1+\alpha}m(s)ds]\\
\leq\frac{C}{N^{\alpha}\rho^{\alpha}}(\mathbb{\hat{E}}[\int_{0}^{T}\mid
y_{\rho}(s)-\bar{y}(s)\mid^{2}ds])^{\frac{1+\alpha}{2}}(\mathbb{\hat{E}}%
[\int_{0}^{T}|m(s)|^{\frac{2}{1-\alpha}}ds])^{\frac{1-\alpha}{2}}\\
\leq\frac{C}{N^{\alpha}}\rho,
\end{array}
\]%
\[%
\begin{array}
[c]{l}%
\mathbb{\hat{E}}[|\int_{0}^{T}\int\nolimits_{0}^{1}(\tilde{f}_{y}%
(s)-f_{y}(s))(I_{\{|\tilde{x}_{\rho}(s)+\hat{x}(s)|>N\}}+I_{\{|u(s)-\bar
{u}(s)|>N\}})d\lambda(y_{\rho}(s)-\bar{y}(s))m(s)ds|]\\
\leq C\mathbb{\hat{E}}[\int_{0}^{T}(I_{\{|\tilde{x}_{\rho}(s)+\hat{x}%
(s)|>N\}}+I_{\{|u(s)-\bar{u}(s)|>N\}})\mid y_{\rho}(s)-\bar{y}(s)\mid
m(s)ds]\\
\leq\frac{C}{N^{\alpha}}\mathbb{\hat{E}}[\int_{0}^{T}(|\tilde{x}_{\rho
}(s)+\hat{x}(s)|^{\alpha}+|u(s)-\bar{u}(s)|^{\alpha})\mid y_{\rho}(s)-\bar
{y}(s)\mid m(s)ds]\\
\leq\frac{C}{N^{\alpha}}(\mathbb{\hat{E}}[\int_{0}^{T}(|\tilde{x}_{\rho
}(s)+\hat{x}(s)|^{2}+|u(s)-\bar{u}(s)|^{2})ds])^{\frac{\alpha}{2}%
}(\mathbb{\hat{E}}[\mid y_{\rho}(s)-\bar{y}(s)\mid^{2}])^{\frac{1}{2}%
}(\mathbb{\hat{E}}[\mid m(s)\mid^{\frac{2}{1-\alpha}}])^{\frac{1-\alpha}{2}}\\
\leq\frac{C}{N^{\alpha}}\rho
\end{array}
\]
and%
\[%
\begin{array}
[c]{l}%
\mathbb{\hat{E}}[|\int_{0}^{T}\int\nolimits_{0}^{1}(\tilde{f}_{y}%
(s)-f_{y}(s))I_{\{|\tilde{x}_{\rho}(s)+\hat{x}(s)|\leq N\} \cap\{ \mid
y_{\rho}(s)-\bar{y}(s)\mid\leq N\rho\} \cap\{|u(s)-\bar{u}(s)|\leq
N\}}d\lambda(y_{\rho}(s)-\bar{y}(s))m(s)ds|]\\
\leq C\mathbb{\hat{E}}[\int_{0}^{T}\bar{\omega}(3N\rho)\mid y_{\rho}%
(s)-\bar{y}(s)\mid m(s)ds]\\
\leq C\bar{\omega}(3N\rho)(\mathbb{\hat{E}}[\int_{0}^{T}\mid y_{\rho}%
(s)-\bar{y}(s)\mid^{2}ds])^{\frac{1}{2}}(\mathbb{\hat{E}}[\int_{0}%
^{T}|m(s)|^{2}ds])^{\frac{1}{2}}\\
\leq C\bar{\omega}(3N\rho)\rho.
\end{array}
\]
Thus we get for each $N>0$,
\begin{align*}
\mathbb{\hat{E}}[  &  \mid\int_{0}^{T}\int\nolimits_{0}^{1}(\tilde{f}%
_{y}(s)-f_{y}(s))d\lambda(y_{\rho}(s)-\bar{y}(s))m(s)ds\mid]\\
&  \leq C\bar{\omega}(3N\rho)\rho+\frac{C}{N^{\alpha}}\rho,
\end{align*}
which easily implies that $\mathbb{\hat{E}}[\mid\int_{0}^{T}\int%
\nolimits_{0}^{1}(\tilde{f}_{y}(s)-f_{y}(s))d\lambda(y_{\rho}(s)-\bar
{y}(s))m(s)ds\mid]=o(\rho)$.

The proof is complete.
\end{proof}

\textbf{Proof of Theorem \ref{variational derivative-y}.}

\textbf{ Step 1.} We first prove that $\underset{\rho\rightarrow0}{\lim}%
\frac{y_{\rho}(0)-\bar{y}(0)}{\rho}$ exists.

Consider%
\[%
\begin{array}
[c]{rl}%
y_{\rho}(t)-\bar{y}(t)= & \phi(x_{\rho}(T))-\phi(\bar{x}(T))+\int_{t}%
^{T}(f_{\rho}(s)-f(s))ds-\int_{t}^{T}(z_{\rho}(s)-\bar{z}(s))dB(s)\\
& -(K_{\rho}(T)-K_{\rho}(t))+(\bar{K}(T)-\bar{K}(t)).
\end{array}
\]
It yields that%
\[%
\begin{array}
[c]{rl}%
\bar{K}(t)+y_{\rho}(t)-\bar{y}(t)= & \bar{K}(T)+\phi_{x}(\bar{x}(T))\rho
\hat{x}(T)+J_{1}-\int_{t}^{T}(z_{\rho}(s)-\bar{z}(s))dB(s)-(K_{\rho
}(T)-K_{\rho}(t))\\
& +\int_{t}^{T}[f_{x}(s)(x_{\rho}(s)-\bar{x}(s))+f_{y}(s)(y_{\rho}(s)-\bar
{y}(s))+f_{u}(s)\rho(u(s)-\bar{u}(s))+J_{2}(s)]ds.
\end{array}
\]
Applying It\^{o}'s formula to $m(t)(\bar{K}(t)+y_{\rho}(t)-\bar{y}(t))$, we
can get
\begin{equation}%
\begin{array}
[c]{rl}%
y_{\rho}(0)-\bar{y}(0)= & \mathbb{\hat{E}[}(\bar{K}(T)+\phi_{x}(\bar
{x}(T))\rho\hat{x}(T)+J_{1})m(T)\\
& +\int_{t}^{T}(f_{x}(s)(x_{\rho}(s)-\bar{x}(s))-f_{y}(s)\bar{K}%
(s)+f_{u}(s)\rho(u(s)-\bar{u}(s))+J_{2}(s))m(s)ds].
\end{array}
\label{v-equation-y-1}%
\end{equation}
Note that%
\[
\bar{K}(T)m(T)=\int_{0}^{T}f_{y}(s)\bar{K}(s)m(s)ds+\int_{0}^{T}m(s)d\bar
{K}(s),
\]
then (\ref{v-equation-y-1}) becomes%
\[%
\begin{array}
[c]{rl}%
y_{\rho}(0)-\bar{y}(0)= & \mathbb{\hat{E}[}(\phi_{x}(\bar{x}(T))\rho\hat
{x}(T)+J_{1})m(T)+\int_{0}^{T}m(s)d\bar{K}(s)\\
& +\int_{t}^{T}(f_{x}(s)(x_{\rho}(s)-\bar{x}(s))+f_{u}(s)\rho(u(s)-\bar
{u}(s))+J_{2}(s))m(s)ds].
\end{array}
\]
By Lemma \ref{variational derivative-backward},%
\begin{equation}%
\begin{array}
[c]{rl}%
y_{\rho}(0)-\bar{y}(0)= & \mathbb{\hat{E}[}\phi_{x}(\bar{x}(T))\rho\hat
{x}(T)m(T)+\int_{0}^{T}m(s)d\bar{K}(s)\\
& +\int_{0}^{T}(f_{x}(s)\rho\hat{x}(s)+f_{u}(s)\rho(u(s)-\bar{u}%
(s)))m(s)ds]+o(\rho).
\end{array}
\label{v-equation-y-2}%
\end{equation}
Since $\bar{u}(\cdot)$ is an optimal control, we have%
\begin{equation}
\frac{y_{\rho}(0)-\bar{y}(0)}{\rho}=\mathbb{\hat{E}[}\frac{\int_{0}%
^{T}m(s)d\bar{K}(s)}{\rho}+\Theta^{u}]+o(1)\geq0. \label{v-equation-y-2-1}%
\end{equation}
Note that $\frac{\int_{0}^{T}m(s)d\bar{K}(s)}{\rho}$ decreases as
$\rho\downarrow0$. It yields that $\mathbb{\hat{E}[}\frac{\int_{0}%
^{T}m(s)d\bar{K}(s)}{\rho}+\Theta^{u}]$ decreases. Since $\mathbb{\hat
{E}[\cdot]}$ is sublinear,
\[
\mathbb{\hat{E}[}\frac{\int_{0}^{T}m(s)d\bar{K}(s)}{\rho}+\Theta^{u}%
]\geq\mathbb{\hat{E}[}\frac{\int_{0}^{T}m(s)d\bar{K}(s)}{\rho}]-\mathbb{\hat
{E}[}-\Theta^{u}]=-\mathbb{\hat{E}[}-\Theta^{u}].
\]
Thus, the limit of $\frac{y_{\rho}(0)-\bar{y}(0)}{\rho}$ exists as
$\rho\rightarrow0$.

\textbf{Step 2.} Then, we prove that there exists a $P^{u}\in\mathcal{P}$ such
that $E_{P^{u}}[\bar{K}(T)]=0$.

Since $\mathcal{P}$ is weakly compact and $\int_{0}^{T}m(s)d\bar{K}%
(s)+\rho\Theta^{u}\in L_{G}^{2}(\Omega_{T})$, there exists a $P^{\rho,u}%
\in\mathcal{P}$ which depends on $\rho$ and $u(\cdot)$ such that%
\[
\mathbb{\hat{E}[}\frac{\int_{0}^{T}m(s)d\bar{K}(s)}{\rho}+\Theta
^{u}]=E_{P^{\rho,u}}\mathbb{[}\frac{\int_{0}^{T}m(s)d\bar{K}(s)}{\rho}%
+\Theta^{u}].
\]
Thus (\ref{v-equation-y-2-1}) becomes%
\begin{equation}
\frac{y_{\rho}(0)-\bar{y}(0)}{\rho}=E_{P^{\rho,u}}\mathbb{[}\frac{\int_{0}%
^{T}m(s)d\bar{K}(s)}{\rho}+\Theta^{u}]+o(1)\geq0. \label{v-equation-y-3}%
\end{equation}
Obviously, there exist a $P^{u}\in\mathcal{P}$ and a sequence $P^{\rho_{n}%
,u}\rightarrow P^{u}$ weakly as $\rho_{n}\rightarrow0$. By
(\ref{v-equation-y-3}), we get%
\[
E_{P^{\rho_{n},u}}\mathbb{[}\int_{0}^{T}m(s)d\bar{K}(s)]=y_{\rho_{n}}%
(0)-\bar{y}(0)-\rho_{n}E_{P^{\rho_{n},u}}\mathbb{[}\Theta^{u}]-o(\rho_{n}).
\]
Note that
\[
\mid E_{P^{\rho_{n},u}}\mathbb{[}\Theta^{u}]\mid\leq\mathbb{\hat{E}[\mid
}\Theta^{u}\mid]<\infty\text{,}%
\]
it yields that $E_{P^{\rho_{n},u}}\mathbb{[}\int_{0}^{T}m(s)d\bar
{K}(s)]\rightarrow0$ as $n\rightarrow\infty$. Since $\int_{0}^{T}m(s)d\bar
{K}(s)$ belongs to $L_{G}^{2}(\Omega_{T})$, it is easy to see that%
\[
E_{P^{\rho_{n},u}}[\int_{0}^{T}m(s)d\bar{K}(s)]\rightarrow E_{P^{u}}[\int%
_{0}^{T}m(s)d\bar{K}(s)]\text{.}%
\]
Thus we deduce that $E_{P^{u}}[\bar{K}(T)]=0$.

\textbf{Step 3.} At last, we prove that $\underset{\rho\rightarrow0}{\lim
}\frac{y_{\rho}(0)-\bar{y}(0)}{\rho}=E_{P^{u}}\mathbb{[}\Theta^{u}%
]=\underset{P\in\mathcal{P}^{\ast}}{\sup}E_{P}[\Theta^{u}].$

By (\ref{v-equation-y-3}) and $\int_{0}^{T}m(s)d\bar{K}(s)\leq0$,
\[
\frac{y_{\rho}(0)-\bar{y}(0)}{\rho}\leq E_{P^{\rho,u}}\mathbb{[}\Theta
^{u}]+o(1).
\]
Then%
\begin{equation}
\underset{\rho\rightarrow0}{\lim}\frac{y_{\rho}(0)-\bar{y}(0)}{\rho}%
\leq\underset{n\rightarrow\infty}{\lim}E_{P^{\rho_{n},u}}\mathbb{[}\Theta
^{u}]=E_{P^{u}}\mathbb{[}\Theta^{u}]. \label{v-equation-y-3-1}%
\end{equation}
For any $P\in\mathcal{P}^{\ast}$, by (\ref{v-equation-y-2-1}), we have%
\[%
\begin{array}
[c]{rl}%
\frac{y_{\rho}(0)-\bar{y}(0)}{\rho} & =\mathbb{\hat{E}[}\frac{\int_{0}%
^{T}m(s)d\bar{K}(s)}{\rho}+\Theta^{u}]+o(1)\\
& \geq E_{P}\mathbb{[}\frac{\int_{0}^{T}m(s)d\bar{K}(s)}{\rho}+\Theta
^{u}]+o(1)\\
& =E_{P}\mathbb{[}\Theta^{u}]+o(1).
\end{array}
\]
It yields that
\begin{equation}
\underset{\rho\rightarrow0}{\lim}\frac{y_{\rho}(0)-\bar{y}(0)}{\rho}\geq
E_{P}\mathbb{[}\Theta^{u}],\; \forall P\in\mathcal{P}^{\ast}.
\label{v-equation-y-3-2}%
\end{equation}
Note that $P^{u}\in\mathcal{P}^{\ast}$, then, by (\ref{v-equation-y-3-1}) and
(\ref{v-equation-y-3-2}), we obtain%
\[
\underset{\rho\rightarrow0}{\lim}\frac{y_{\rho}(0)-\bar{y}(0)}{\rho}=E_{P^{u}%
}\mathbb{[}\Theta^{u}]=\underset{P\in\mathcal{P}^{\ast}}{\sup}E_{P}[\Theta
^{u}].
\]
This completes the proof. $\Box$

\subsection{Variational inequality}

We obtain the following variational inequality.

\begin{theorem}
\label{variational ineq} Suppose (H1)-(H3) hold. Then there exists a $P^{\ast
}\in\mathcal{P}^{\ast}$ such that%
\[
\underset{u\in\mathcal{U}[0,T]}{\inf}E_{P^{\ast}}[\Theta^{u}]\geq0.
\]

\end{theorem}

\begin{proof}
By Theorem \ref{variational derivative-y}, we can get for any $u(\cdot
)\in\mathcal{U}[0,T]$,
\[
\underset{\rho\rightarrow0}{\lim}\frac{y_{\rho}(0)-\bar{y}(0)}{\rho}%
=\sup_{P\in\mathcal{P}^{\ast}}E_{P}[\Theta^{u}]\geq0.
\]
Then
\[
\underset{u\in\mathcal{U}[0,T]}{\inf}\sup_{P\in\mathcal{P}^{\ast}}E_{P}%
[\Theta^{u}]\geq0.
\]

It is easy to check that $\mathcal{P}^{\ast}$ is convex and weakly compact,
and for $\lambda\in\lbrack0,1]$, $u$, $u^{\prime}\in\mathcal{U}[0,T]$,%
\[
\Theta^{\lambda u+(1-\lambda)u^{\prime}}=\lambda\Theta^{u}+(1-\lambda
)\Theta^{u^{\prime}}.
\]
Thus, by Sion's minimax theorem, we obtain%
\[
\underset{u\in\mathcal{U}[0,T]}{\inf}\sup_{P\in\mathcal{P}^{\ast}}E_{P}%
[\Theta^{u}]=\sup_{P\in\mathcal{P}^{\ast}}\underset{u\in\mathcal{U}%
[0,T]}{\inf}E_{P}[\Theta^{u}].
\]
Then, for each $\varepsilon>0,$ there exists a $P^{\varepsilon}\in
\mathcal{P}^{\ast}$ such that
\[
\underset{u\in\mathcal{U}[0,T]}{\inf}E_{P^{\varepsilon}}[\Theta^{u}%
]\geq-\varepsilon.
\]
Since $\mathcal{P}^{\ast}$ is weakly compact, there exist a $P^{\ast}%
\in\mathcal{P}^{\ast}$ and a sequence $P^{\varepsilon_{n}}\rightarrow P^{\ast
}$ weakly as $\varepsilon_{n}\rightarrow0$. Note that for any $u(\cdot
)\in\mathcal{U}[0,T]$,
\[
E_{P^{\varepsilon_{n}}}[\Theta^{u}]\geq-\varepsilon_{n}.
\]
Letting $\varepsilon_{n}\rightarrow0$, it yields that for any $u(\cdot
)\in\mathcal{U}[0,T]$,%
\[
E_{P^{\ast}}[\Theta^{u}]\geq0.
\]
Thus, we have
\[
\underset{u\in\mathcal{U}[0,T]}{\inf}E_{P^{\ast}}[\Theta^{u}]\geq0.
\]
This completes the proof.
\end{proof}

\subsection{Maximum principle}

Consider the following kind of BSDE under $P^{\ast}$:%
\begin{equation}
\left\{
\begin{array}
[c]{rl}%
-dp(t)= & [(f_{x}(t))^{T}+(b_{x}(t))^{T}p(t)+f_{y}(t)p(t)]dt+(\sigma_{x}%
^{i}(t))^{T}q^{j}(t)d\langle B^{i},B^{j}\rangle(t)-q^{i}(t)dB^{i}(t)-dN(t),\\
p(T)= & (\phi_{x}(\bar{x}(T)))^{T},
\end{array}
\right.  \label{adjoint neweq-1}%
\end{equation}
where $(p(t))_{t\in\lbrack0,T]}\in M_{P^{\ast}}^{2}(0,T;\mathbb{R}^{n})=\{
\eta:\eta$ is $\mathbb{R}^{n}$-valued progressively measurable and
$E_{P^{\ast}}[\int_{0}^{T}|\eta_{t}|^{2}dt]<\infty\}$, $(q(t))_{t\in
\lbrack0,T]}\in M_{P^{\ast}}^{2}(0,T;\mathbb{R}^{n\times d})$, $(N_{t}%
)_{t\in\lbrack0,T]}\in\mathcal{M}_{P^{\ast}}^{2,\perp}(0,T;\mathbb{R}%
^{n}):=\{N:$ all $\mathbb{R}^{n}$-valued square integrable martingale that is
orthogonal to $B\}$.

\begin{remark}
Note that $B$ is only a continuous martingale under $P^{\ast}$ and the
martingale representation theorem may not hold. So it is necessary to
introduce the third term $N$ which is orthogonal to $B$.
\end{remark}

Following El Karoui and Huang \cite{EM} and Buckdahn et. al. \cite{BLRT},
there exists a unique $(p(\cdot),q(\cdot),N(\cdot))\in M_{P^{\ast}}%
^{2}(0,T;\mathbb{R}^{n})\times M_{P^{\ast}}^{2}(0,T;\mathbb{R}^{n\times
d})\times\mathcal{M}_{P^{\ast}}^{2,\perp}(0,T;\mathbb{R}^{n})$ which solves
the adjoint equation (\ref{adjoint neweq-1}). Applying It\^{o}'s formula to
$\langle\hat{x}(t),m(t)p(t)\rangle$, we obtain%
\begin{align*}
&  E_{P^{\ast}}[\phi_{x}(\bar{x}(T))\hat{x}(T)m(T)+\int_{0}^{T}f_{x}(s)\hat
{x}(s)m(s)ds]\\
&  =E_{P^{\ast}}[\int_{0}^{T}(m(t)\langle p(t),b_{u}(t)(u(t)-\bar
{u}(t))\rangle+m(t)\langle q^{j}(t),\sigma_{u}^{i}(t)(u(t)-\bar{u}%
(t))\rangle\gamma^{ij}(t))dt],
\end{align*}
where $\Gamma(t)=(\gamma^{ij}(t))$, $d\langle B^{i},B^{j}\rangle
(t)=\gamma^{ij}(t)dt$. We define the Hamiltonian $H:\mathbb{R}^{n}%
\times\mathbb{R}\times\mathbb{R}^{m}\times\mathbb{R}^{n}\times\mathbb{R}%
^{n\times d}\times\lbrack0,T]\rightarrow\mathbb{R}$ as follows:%
\[
H(x,y,u,p,q,t)=\langle p,b(t,x,u)\rangle+\langle q^{j},\sigma^{i}%
(t,x,u)\rangle\gamma^{ij}(t)+f(t,x,y,u).
\]
Thus%
\[%
\begin{array}
[c]{rl}%
E_{P^{\ast}}[\Theta^{u}] & =E_{P^{\ast}}[\int_{0}^{T}m(t)\langle(b_{u}%
(t))^{T}p(t)+(f_{u}(t))^{T}+(\sigma_{u}^{i}(t))^{T}q^{j}(t)\gamma
^{ij}(t),u(t)-\bar{u}(t)\rangle dt]\\
& =E_{P^{\ast}}[\int_{0}^{T}m(t)\langle(H_{u}(\bar{x}(t),\bar{y}(t),\bar
{u}(t),p(t),q(t),t))^{T},u(t)-\bar{u}(t)\rangle dt]\\
& =E_{P^{\ast}}[\int_{0}^{T}m(t)H_{u}(\bar{x}(t),\bar{y}(t),\bar
{u}(t),p(t),q(t),t)(u(t)-\bar{u}(t))dt].
\end{array}
\]
By Theorem \ref{variational ineq}, $E_{P^{\ast}}[\Theta^{u}]\geq0$ for each
$u(\cdot)\in\mathcal{U}[0,T]$, then we can get%
\begin{equation}
H_{u}(\bar{x}(t),\bar{y}(t),\bar{u}(t),p(t),q(t),t)(u-\bar{u}(t))\geq0,\;
\forall u\in U,\;a.e.,\;P^{\ast}-a.s.. \label{MP}%
\end{equation}
We summarize the above analysis to the following stochastic maximum principle.

\begin{theorem}
\label{Thm-MP-new} Suppose (H1)-(H3) hold. Let $\bar{u}(\cdot)$ be an optimal
control and $(\bar{x}(\cdot),\bar{y}(\cdot),\bar{z}(\cdot),\bar{K}(\cdot))$ be
the corresponding trajectory. Then there exist a $P^{\ast}\in\mathcal{P}%
^{\ast}$ and $(p(\cdot),q(\cdot),N(\cdot))\in M_{P^{\ast}}^{2}(0,T;\mathbb{R}%
^{n})\times M_{P^{\ast}}^{2}(0,T;\mathbb{R}^{n\times d})\times\mathcal{M}%
_{P^{\ast}}^{2,\perp}(0,T;\mathbb{R}^{n})$, which is the solution of the
adjoint equation (\ref{adjoint neweq-1}), such that the inequality (\ref{MP}) holds.
\end{theorem}

\subsection{Sufficient condition}

In this subsection, we give the sufficient condition for optimality.

\begin{theorem}
\label{the-sufficient} Suppose (H1)-(H3) hold. Let $\bar{u}(\cdot
)\in\mathcal{U}[0,T]$ and $P^{\ast}\in\mathcal{P}^{\ast}$ satisfy that
\[
H_{u}(\bar{x}(t),\bar{y}(t),\bar{u}(t),p(t),q(t),t)(u-\bar{u}(t))\geq0,\;
\forall u\in U,\;a.e.,\;P^{\ast}-a.s.,
\]
where $(\bar{x}(\cdot),\bar{y}(\cdot),\bar{z}(\cdot),\bar{K}(\cdot))$ is the
state processes of (\ref{state-1}) and (\ref{state-2}) corresponding to
$\bar{u}(\cdot)$ and $(p(\cdot),q(\cdot),N(\cdot))$ is the solution of the
adjoint equation (\ref{adjoint neweq-1}) under $P^{\ast}$. We also assume that
$H$ is convex with respect to $x$, $y$, $u$ and $\phi$ is convex with respect
to $x$. Then $\bar{u}(\cdot)$ is an optimal control.
\end{theorem}

\begin{proof}
For any $u(\cdot)\in\mathcal{U}[0,T]$, let $(x(\cdot),y(\cdot),z(\cdot
),K(\cdot))$ be the corresponding state processes of (\ref{state-1}) and
(\ref{state-2}). Define $\xi(t):=x(t)-\bar{x}(t)$ and $\eta(t):=y(t)-\bar
{y}(t)$. Then $\xi(\cdot)$ and $\eta(\cdot)$ satisfy the following equations
under $P^{\ast}$:%
\[
\left\{
\begin{array}
[c]{rl}%
d\xi(t)= & [b_{x}(t)\xi(t)+\alpha(t)]dt+[\sigma_{x}^{i}(t)\xi(t)+\beta
^{i}(t)]dB^{i}(t),\\
\xi(0)= & 0,
\end{array}
\right.
\]
where
\begin{align*}
\alpha(t)  &  :=-b_{x}(t)\xi(t)+b(t,x(t),u(t))-b(t,\bar{x}(t),\bar{u}(t)),\\
\beta^{i}(t)  &  :=-\sigma_{x}^{i}(t)\xi(t)+\sigma^{i}(t,x(t),u(t))-\sigma
^{i}(t,\bar{x}(t),\bar{u}(t)),
\end{align*}
and%
\[
\left\{
\begin{array}
[c]{rl}%
-d\eta(t)= & [f_{x}(t)\xi(t)+f_{y}(t)\eta(t)+\tilde{\alpha}(t)]dt-\tilde
{z}(t)dB(t)-dK(t),\\
\eta(T)= & \phi_{x}(\bar{x}(T))\xi(T)+\tilde{\beta}(T),
\end{array}
\right.
\]
where $\tilde{z}(t):=z(t)-\bar{z}(t)$,
\begin{align*}
\tilde{\alpha}(t)  &  :=-f_{x}(t)\xi(t)-f_{y}(t)\eta
(t)+f(t,x(t),y(t),u(t))-f(t,\bar{x}(t),\bar{y}(t),\bar{u}(t)),\\
\tilde{\beta}(T)  &  :=-\phi_{x}(\bar{x}(T))\xi(T)+\phi(x(T))-\phi(\bar
{x}(T)).
\end{align*}
For simplicity, set
\[
H(t):=H(\bar{x}(t),\bar{y}(t),\bar{u}(t),p(t),q(t),t).
\]
The definitions of $H_{x}(t)$, $H_{y}(t)$ and $H_{u}(t)$ are similar. Applying
It\^{o}'s lemma to $\langle\xi(t),m(t)p(t)\rangle-\eta(t)m(t)$ under $P^{\ast
}$, we can derive%
\[%
\begin{array}
[c]{l}%
E_{P^{\ast}}[\langle\xi(T),m(T)p(T)\rangle-\langle\xi(0),m(0)p(0)\rangle
-\eta(T)m(T)+\eta(0)m(0)]\\
=E_{P^{\ast}}[\int\nolimits_{0}^{T}(\langle m(t)p(t),\alpha(t)\rangle+\langle
m(t)q^{j}(t),\beta^{i}(t)\rangle\gamma^{ij}(t)+m(t)\tilde{\alpha}%
(t))dt-\int\nolimits_{0}^{T}m(t)dK(t)]\\
=E_{P^{\ast}}[\int\nolimits_{0}^{T}(-H_{x}(t)\xi(t)-H_{y}(t)\eta
(t)+H(x(t),y(t),u(t),p(t),q(t),t)-H(t))m(t)dt-\int\nolimits_{0}^{T}%
m(t)dK(t)]\\
\geq E_{P^{\ast}}\int\nolimits_{0}^{T}[-H_{x}(t)\xi(t)-H_{y}(t)\eta
(t)-H_{u}(t)(u(t)-\bar{u}(t))+H(x(t),y(t),u(t),p(t),q(t),t)-H(t)]m(t)dt.
\end{array}
\]
The last inequality is due to the assumption and $-m(t)dK(t)\geq0$. Note that
$H$ is convex with respect to $x$, $y$, $u$. We have%
\[
-H_{x}(t)\xi(t)-H_{y}(t)\eta(t)-H_{u}(t)(u(t)-\bar{u}(t))\geq
H(t)-H(x(t),y(t),u(t),p(t),q(t),m(t),t).
\]
It yields that%
\[
E_{P^{\ast}}[\langle\xi(T),m(T)p(T)\rangle-\langle\xi(0),m(0)p(0)\rangle
-\eta(T)m(T)+\eta(0)m(0)]\geq0,
\]
which leads to $E_{P^{\ast}}[-\tilde{\beta}(T)m(T)+\eta(0)]\geq0$. Since
$\phi$ is convex with respect to $x$, we have that $\tilde{\beta}(T)\geq0$.
Thus, $\eta(0)\geq0$, which implies that $\bar{u}(\cdot)$ is an optimal
control. This completes the proof.\bigskip
\end{proof}

\section{The general case}

In this section, we consider the general state equations.\bigskip

\subsection{$f$ includes $z$ term}

Now we study the case in which the generator $f$ of (\ref{state-2}) includes
the term $z$ and we use the notations in Section 4. For simplicity, we assume
that $f$ only contains the term $z$, the other terms can be analyzed similarly
as\ in Section 4. Similar to the proof of Theorem
\textbf{\ref{variational derivative-y}}, we can get%
\[%
\begin{array}
[c]{rl}%
\bar{K}(t)+y_{\rho}(t)-\bar{y}(t)= & \bar{K}(T)+\phi_{x}(\bar{x}(T))\rho
\hat{x}(T)+J_{1}+\int_{t}^{T}A_{\rho}(s)(z_{\rho}(s)-\bar{z}(s))^{T}ds\\
& -\int_{t}^{T}(z_{\rho}(s)-\bar{z}(s))dB(s)-(K_{\rho}(T)-K_{\rho}(t)),
\end{array}
\]
where $J_{1}$ is the same as in Section 4 and $A_{\rho}(s):=\int_{0}^{1}%
f_{z}(\bar{z}(s)+\lambda(z_{\rho}(s)-\bar{z}(s)))d\lambda$. Following
\cite{HJPS1}, we construct an auxiliary extended $\tilde{G}$-expectation space
$(\tilde{\Omega},L_{\tilde{G}}^{1}(\tilde{\Omega}),\mathbb{\hat{E}}^{\tilde
{G}})$ with $\tilde{\Omega}=C_{0}([0,\infty),\mathbb{R}^{2d})$ and%

\[
\tilde{G}(A)=\frac{1}{2}\sup_{\gamma\in\Gamma}\mathrm{tr}\left[  A\left[
\begin{array}
[c]{cc}%
\gamma\gamma^{T} & I\\
I & (\gamma\gamma^{T})^{-1}%
\end{array}
\right]  \right]  ,\ A\in\mathbb{S}_{2d}.
\]
Let $(B(t),\tilde{B}(t))_{t\geq0}$ be the canonical process in the extended
space. It is easy to check that $\langle B^{i},\tilde{B}^{j}\rangle
(t)=\delta_{ij}t$. Consider the equation%
\[
dm_{\rho}(t)=A_{\rho}(t)m_{\rho}(t)d\tilde{B}(t),\text{ }m_{\rho}(0)=1.
\]
Applying It\^{o}'s formula to $m_{\rho}(t)(\bar{K}(t)+y_{\rho}(t)-\bar{y}%
(t))$, we can get
\begin{equation}%
\begin{array}
[c]{rl}%
y_{\rho}(0)-\bar{y}(0)= & \mathbb{\hat{E}}^{\tilde{G}}\mathbb{[}(\bar
{K}(T)+\phi_{x}(\bar{x}(T))\rho\hat{x}(T)+J_{1})m_{\rho}(T)].
\end{array}
\label{new-z-term}%
\end{equation}
Note that%
\[
\bar{K}(T)m_{\rho}(T)=\int_{0}^{T}A_{\rho}(s)\bar{K}(s)m_{\rho}(s)d\tilde
{B}(s)+\int_{0}^{T}m_{\rho}(s)d\bar{K}(s),
\]
then (\ref{new-z-term}) becomes%
\[%
\begin{array}
[c]{rl}%
y_{\rho}(0)-\bar{y}(0) & =\mathbb{\hat{E}}^{\tilde{G}}\mathbb{[}(\phi_{x}%
(\bar{x}(T))\rho\hat{x}(T)+J_{1})m_{\rho}(T)+\int_{0}^{T}m_{\rho}(s)d\bar
{K}(s)]\\
& =\mathbb{\hat{E}}^{\tilde{G}}\mathbb{[}(\phi_{x}(\bar{x}(T))\rho\hat
{x}(T)m(T)+\int_{0}^{T}m_{\rho}(s)d\bar{K}(s)+J_{1}m_{\rho}(T)+J_{2}],
\end{array}
\]
where $J_{2}=\phi_{x}(\bar{x}(T))\rho\hat{x}(T)(m_{\rho}(T)-m(T))$ and%
\[
dm(t)=f_{z}(t)m(t)d\tilde{B}(t),\text{ }m(0)=1.
\]
Similar to the proof of Lemma \ref{variational derivative-backward}, we can
obtain $\mathbb{\hat{E}}^{\tilde{G}}\mathbb{[}\mid J_{1}m_{\rho}%
(T)\mid]=o(\rho)$. By Proposition 3.8 in \cite{HJPS}, we can get%
\[
\mathbb{\hat{E}}^{\tilde{G}}\mathbb{[}\int_{0}^{T}|z_{\rho}(s)-\bar{z}%
(s)|^{2}ds]\leq C\rho.
\]
Then similar to the proof of Proposition \ref{variational derivative-forward},
we can easily obtain $\mathbb{\hat{E}}^{\tilde{G}}[|\hat{x}(T)(m_{\rho
}(T)-m(T))|]=o(1)$. Thus we get%
\[%
\begin{array}
[c]{rl}%
\frac{y_{\rho}(0)-\bar{y}(0)}{\rho}= & \mathbb{\hat{E}}^{\tilde{G}}%
\mathbb{[}\frac{\int_{0}^{T}m_{\rho}(s)d\bar{K}(s)}{\rho}+\phi_{x}(\bar
{x}(T))\hat{x}(T)m(T)]+o(1).
\end{array}
\]
We can choose a sequence $\rho_{k}\downarrow0$ such that $\tilde{P}^{k,u}%
\in\mathcal{\tilde{P}}$ converges weakly to $\tilde{P}^{u}\in\mathcal{\tilde
{P}}$ and
\[
\lim_{k\rightarrow\infty}\frac{y_{\rho_{k}}(0)-\bar{y}(0)}{\rho_{k}}%
=\limsup_{\rho\rightarrow0}\frac{y_{\rho}(0)-\bar{y}(0)}{\rho},
\]%
\[
\mathbb{\hat{E}}^{\tilde{G}}\mathbb{[}\frac{\int_{0}^{T}m_{\rho_{k}}%
(s)d\bar{K}(s)}{\rho_{k}}+\phi_{x}(\bar{x}(T))\hat{x}(T)m(T)]=E_{\tilde
{P}^{k,u}}\mathbb{[}\frac{\int_{0}^{T}m_{\rho_{k}}(s)d\bar{K}(s)}{\rho_{k}%
}+\phi_{x}(\bar{x}(T))\hat{x}(T)m(T)],
\]
where $\mathcal{\tilde{P}}$ represents $\mathbb{\hat{E}}^{\tilde{G}%
}\mathbb{[\cdot]}$. It is easy to check that $E_{\tilde{P}^{k,u}}%
\mathbb{[}\int_{0}^{T}m_{\rho_{k}}(s)d\bar{K}(s)]\rightarrow0$ as
$k\rightarrow\infty$. Note that
\[
\mathbb{\hat{E}}^{\tilde{G}}\mathbb{[}|\int_{0}^{T}(m_{\rho}(s)-m(s))d\bar
{K}(s)|]\rightarrow0\text{ as }\rho\rightarrow0,
\]
then we can get $E_{\tilde{P}^{k,u}}\mathbb{[}\int_{0}^{T}m(s)d\bar
{K}(s)]\rightarrow0$ as $k\rightarrow\infty$. Similar to the proof of Theorem
\textbf{\ref{variational derivative-y}}, we can get $\tilde{P}^{u}%
\in\mathcal{\tilde{P}}^{\ast}=\{ \tilde{P}\in\mathcal{\tilde{P}}:E_{\tilde{P}%
}[\bar{K}(T)]=0\}$ and
\[
\sup_{\tilde{P}\in\mathcal{\tilde{P}}^{\ast}}E_{\tilde{P}}\mathbb{[}\phi
_{x}(\bar{x}(T))\hat{x}(T)m(T)]\leq\liminf_{\rho\rightarrow0}\frac{y_{\rho
}(0)-\bar{y}(0)}{\rho}\leq\limsup_{\rho\rightarrow0}\frac{y_{\rho}(0)-\bar
{y}(0)}{\rho}\leq E_{\tilde{P}^{u}}\mathbb{[}\phi_{x}(\bar{x}(T))\hat
{x}(T)m(T)],
\]
which implies
\[
\lim_{\rho\rightarrow0}\frac{y_{\rho}(0)-\bar{y}(0)}{\rho}=E_{\tilde{P}^{u}%
}\mathbb{[}\phi_{x}(\bar{x}(T))\hat{x}(T)m(T)].
\]
Similar to the proof of Theorem \ref{variational ineq}, there exists a
$\tilde{P}^{\ast}\in\mathcal{\tilde{P}}^{\ast}$such that
\[
\underset{u\in\mathcal{U}[0,T]}{\inf}E_{\tilde{P}^{\ast}}[\phi_{x}(\bar
{x}(T))\hat{x}(T)m(T)]\geq0.
\]
Now we introduce the following adjoint equation under $\tilde{P}^{\ast}$:%
\begin{equation}
\left\{
\begin{array}
[c]{rl}%
-d\tilde{p}(t)= & \{[(b_{x}(t))^{T}+f_{z_{i}}(t)(\sigma_{x}^{i}(t))^{T}%
]\tilde{p}(t)+f_{z_{i}}(t)\tilde{q}^{1,i}(t)\}dt\\
& +(\sigma_{x}^{i}(t))^{T}\tilde{q}^{1,j}(t)d\langle B^{i},B^{j}%
\rangle(t)+f_{z_{j}}(t)[\tilde{q}^{2,i}(t)-f_{z_{i}}(t)\tilde{p}%
(t)]d\langle\tilde{B}^{i},\tilde{B}^{j}\rangle(t)\\
& -\tilde{q}^{1,i}(t)dB^{i}(t)-[\tilde{q}^{2,i}(t)-f_{z_{i}}(t)\tilde
{p}(t)]d\tilde{B}^{i}(t)-d\tilde{N}(t),\\
\tilde{p}(T)= & (\phi_{x}(\bar{x}(T)))^{T}.
\end{array}
\right.  \label{new-adjoint}%
\end{equation}
Set $\mathcal{F}=\sigma(B_{t}:t\geq0)$ and $P^{\ast}=\tilde{P}^{\ast}%
\mid_{\mathcal{F}}$. We first show that $(\tilde{p}(\cdot),\tilde{q}^{1}%
(\cdot),\tilde{N}(\cdot))\in M_{P^{\ast}}^{2}(0,T;\mathbb{R}^{n})\times
M_{P^{\ast}}^{2}(0,T;\mathbb{R}^{n\times d})\times\mathcal{M}_{P^{\ast}%
}^{2,\perp}(0,T;\mathbb{R}^{n})$. For this we consider the following BSDE
under $(\Omega,\mathcal{F},P^{\ast})$:%
\begin{equation}
\left\{
\begin{array}
[c]{rl}%
-dp(t)= & \{[(b_{x}(t))^{T}+f_{z_{i}}(t)(\sigma_{x}^{i}(t))^{T}]p(t)+f_{z_{i}%
}(t)q^{i}(t)\}dt\\
& +(\sigma_{x}^{i}(t))^{T}q^{j}(t)d\langle B^{i},B^{j}\rangle(t)-q^{i}%
(t)dB^{i}(t)-dN(t),\\
p(T)= & (\phi_{x}(\bar{x}(T)))^{T}.
\end{array}
\right.  \label{neq-adjoint-11}%
\end{equation}
By \cite{EM, BLRT}, the above BSDE has a unique solution $(p(\cdot
),q(\cdot),N(\cdot))\in M_{P^{\ast}}^{2}(0,T;\mathbb{R}^{n})\times M_{P^{\ast
}}^{2}(0,T;\mathbb{R}^{n\times d})\times\mathcal{M}_{P^{\ast}}^{2,\perp
}(0,T;\mathbb{R}^{n})$. It is easy to check that%
\begin{equation}
(\tilde{p}(\cdot),\tilde{q}^{1}(\cdot),\tilde{q}^{2}(\cdot),\tilde{N}%
(\cdot))=(p(\cdot),q(\cdot),p(\cdot)f_{z}(\cdot),N(\cdot)) \label{relation11}%
\end{equation}
is the unique solution of the adjoint equation (\ref{new-adjoint}). Applying
It\^{o}'s formula to $\langle\hat{x}(t),m(t)\tilde{p}(t)\rangle$ under
$\tilde{P}^{\ast}$ and relation (\ref{relation11}), we can get%
\[%
\begin{array}
[c]{l}%
E_{\tilde{P}^{\ast}}[\phi_{x}(\bar{x}(T))\hat{x}(T)m(T)]\\
=E_{\tilde{P}^{\ast}}[\int_{0}^{T}(\langle m(t)p(t),b_{u}(t)(u(t)-\bar
{u}(t))\rangle+\langle m(t)f_{z_{i}}(t)p(t),\sigma_{u}^{i}(t)(u(t)-\bar
{u}(t))\rangle\\
\ \ +\langle m(t)q^{j}(t),\sigma_{u}^{i}(t)(u(t)-\bar{u}(t))\rangle\gamma
^{ij}(t))dt].
\end{array}
\]
We define the Hamiltonian $H:\mathbb{R}^{n}\times\mathbb{R}\times
\mathbb{R}^{1\times d}\times\mathbb{R}^{m}\times\mathbb{R}^{n}\times
\mathbb{R}^{n\times d}\times\lbrack0,T]\rightarrow\mathbb{R}$ as follows:%
\[
H(x,z,u,p,q,t)=\langle p,b(t,x,u)\rangle+\langle f_{z_{i}}(z)p,\sigma
^{i}(t,x,u)\rangle+\langle q^{j},\sigma^{i}(t,x,u)\rangle\gamma^{ij}(t)+f(z).
\]
Then%
\[
E_{\tilde{P}^{\ast}}[\phi_{x}(\bar{x}(T))\hat{x}(T)m(T)]=E_{\tilde{P}^{\ast}%
}[\int_{0}^{T}m(t)H_{u}(\bar{x}(t),\bar{z}(t),\bar{u}%
(t),p(t),q(t),t)(u(t)-\bar{u}(t))dt].
\]
Thus%
\[
H_{u}(\bar{x}(t),\bar{z}(t),\bar{u}(t),p(t),q(t),t)(u-\bar{u}(t))\geq0,\;
\forall u\in U,\;a.e.,\; \tilde{P}^{\ast}-a.s..
\]
Note that all the terms in the above inequality are measurable with respect to
$\mathcal{F}$, then we get%
\begin{equation}
H_{u}(\bar{x}(t),\bar{z}(t),\bar{u}(t),p(t),q(t),t)(u-\bar{u}(t))\geq0,\;
\forall u\in U,\;a.e.,\;P^{\ast}-a.s.. \label{max-pri}%
\end{equation}
We summarize the above analysis to the following theorem.

\begin{theorem}
Suppose (H1)-(H3) hold and $f$ only depends on the term $z$. Let $\bar
{u}(\cdot)$ be an optimal control and $(\bar{x}(\cdot),\bar{y}(\cdot),\bar
{z}(\cdot),\bar{K}(\cdot))$ be the corresponding trajectory. Then there exist
a $P^{\ast}\in\mathcal{P}^{\ast}$ and $(p(\cdot),q(\cdot),N(\cdot))\in
M_{P^{\ast}}^{2}(0,T;\mathbb{R}^{n})\times M_{P^{\ast}}^{2}(0,T;\mathbb{R}%
^{n\times d})\times\mathcal{M}_{P^{\ast}}^{2,\perp}(0,T;\mathbb{R}^{n})$,
which is the solution of the adjoint equation (\ref{neq-adjoint-11}), such
that the inequality (\ref{max-pri}) holds.
\end{theorem}

\bigskip

\subsection{The general maximum principle}

In this subsection, we study the general case, i.e. the state equations are
governed by (\ref{state-1}) and (\ref{state-2}). We only list the main results
since the proofs are similar as in section 4 and subsection 5.1.

For this case, we introduce the following variational equation:%

\[
\left\{
\begin{array}
[c]{rl}%
d\hat{x}(t)= & [b_{x}(t)\hat{x}(t)+b_{u}(t)(u(t)-\bar{u}(t))]dt+[h_{x}%
^{ij}(t)\hat{x}(t)+h_{u}^{ij}(t)(u(t)-\bar{u}(t))]d\langle B^{i},B^{j}%
\rangle(t)\\
& +[\sigma_{x}^{i}(t)\hat{x}(t)+\sigma_{u}^{i}(t)(u(t)-\bar{u}(t))]dB^{i}%
(t),\\
\hat{x}(0)= & 0.
\end{array}
\right.
\]

Similarly, for some $P^{\ast}\in\mathcal{P}^{\ast}$, the following adjoint
equation has a unique solution $(p(\cdot),q(\cdot),N(\cdot))\in M_{P^{\ast}%
}^{2}(0,T;\mathbb{R}^{n})\times M_{P^{\ast}}^{2}(0,T;\mathbb{R}^{n\times
d})\times\mathcal{M}_{P^{\ast}}^{2,\perp}(0,T;\mathbb{R}^{n})$.
\begin{equation}
\left\{
\begin{array}
[c]{rl}%
-dp(t)= & \{(f_{x}(t))^{T}+[(b_{x}(t))^{T}+f_{z_{i}}(t)(\sigma_{x}^{i}%
(t))^{T}+f_{y}(t)]p(t)+f_{z_{i}}(t)q^{l}(t)\}dt\\
& +\{(g_{x}^{ij}(t))^{T}+[(h_{x}^{ij}(t))^{T}+g_{z_{l}}^{ij}(t)(\sigma_{x}%
^{l}(t))^{T}+g_{y}^{ij}(t)]p(t)+g_{z_{l}}^{ij}(t)q^{l}(t)\\
& (\sigma_{x}^{i}(t))^{T}q^{j}(t)\}d\langle B^{i},B^{j}\rangle(t)-q^{i}%
(t)dB^{i}(t)-dN(t),\\
p(T)= & (\phi_{x}(\bar{x}(T)))^{T}.
\end{array}
\right.  \label{neq-adjoint-11-11}%
\end{equation}

Define the Hamiltonian $H:\mathbb{R}^{n}\times\mathbb{R}\times\mathbb{R}%
^{1\times d}\times\mathbb{R}^{m}\times\mathbb{R}^{m}\times\mathbb{R}^{n}%
\times\mathbb{R}^{n\times d}\times\lbrack0,T]\rightarrow\mathbb{R}$ as
follows:%
\[%
\begin{array}
[c]{rl}%
H(x,y,z,u,v,p,q,t)= & \langle p,b(t,x,u)\rangle+\langle p,h^{ij}%
(t,x,u)\rangle\gamma^{ij}(t)+\langle q^{j},\sigma^{i}(t,x,u)\rangle\gamma
^{ij}(t)+\langle p(f_{z_{l}}(t,x,y,z,v)\\
& +g_{z_{l}}^{ij}(t,x,y,z,v)),\sigma^{l}(t,x,u)\rangle\gamma^{ij}%
(t)+f(t,x,y,z,u)+g^{ij}(t,x,y,z,u)\gamma^{ij}(t),
\end{array}
\]
where $i$, $j$, $l=1,\ldots,d$.

\begin{theorem}
Suppose (H1)-(H3) hold. Let $\bar{u}(\cdot)$ be an optimal control and
$(\bar{x}(\cdot),\bar{y}(\cdot),\bar{z}(\cdot),\bar{K}(\cdot))$ be the
corresponding trajectory. Then there exist a $P^{\ast}\in\mathcal{P}^{\ast}$
and $(p(\cdot),q(\cdot),N(\cdot))\in M_{P^{\ast}}^{2}(0,T;\mathbb{R}%
^{n})\times M_{P^{\ast}}^{2}(0,T;\mathbb{R}^{n\times d})\times\mathcal{M}%
_{P^{\ast}}^{2,\perp}(0,T;\mathbb{R}^{n})$, which is the solution of the
adjoint equation (\ref{neq-adjoint-11-11}), such that%
\begin{equation}
H_{u}(\bar{x}(t),\bar{y}(t),\bar{z}(t),\bar{u}(t),\bar{u}%
(t),p(t),q(t),t)(u-\bar{u}(t))\geq0,\; \forall u\in U,\;a.e.,\;P^{\ast}-a.s..
\label{max-pri-11}%
\end{equation}

\end{theorem}

In the following, we give the sufficient condition for optimality.

\begin{theorem}
Suppose (H1)-(H3) hold. Let $\bar{u}(\cdot)\in\mathcal{U}[0,T]$ and $P^{\ast
}\in\mathcal{P}^{\ast}$ satisfy that
\[
H_{u}(\bar{x}(t),\bar{y}(t),\bar{z}(t),\bar{u}(t),\bar{u}%
(t),p(t),q(t),t)(u-\bar{u}(t))\geq0,\; \forall u\in U,\;a.e.,\;P^{\ast}-a.s.,
\]
where $(\bar{x}(\cdot),\bar{y}(\cdot),\bar{z}(\cdot),\bar{K}(\cdot))$ is the
state processes of (\ref{state-1}) and (\ref{state-2}) corresponding to
$\bar{u}(\cdot)$ and $(p(\cdot),q(\cdot),N(\cdot))$ is the solution of the
adjoint equation (\ref{neq-adjoint-11-11}) under $P^{\ast}$. We also assume
that $H$ is convex with respect to $x$, $y$, $z$, $u$ and $\phi$ is convex
with respect to $x$. Then $\bar{u}(\cdot)$ is an optimal control.
\end{theorem}

\section{LQ problem}

For simplicity, we suppose $d=1$. In this case,%
\[
G(a)=\frac{1}{2}(\bar{\sigma}^{2}a^{+}-\underline{\sigma}^{2}a^{-}),\text{
}a\in\mathbb{R},
\]
where $\bar{\sigma}^{2}=\mathbb{\hat{E}}[(B_{1})^{2}]$, $\underline{\sigma
}^{2}=-\mathbb{\hat{E}}[-(B_{1})^{2}]$. Consider the following LQ problem. The
state equation is%
\begin{equation}
\left\{
\begin{array}
[c]{rl}%
dx(t)= & [A(t)x(t)+\tilde{B}(t)u(t)+b(t)]dt+[C(t)x(t)+D(t)u(t)+\sigma
(t)]dB(t),\\
x(0)= & x_{0},\;x_{0}\in\mathbb{R}^{n},
\end{array}
\right.  \label{Linear q}%
\end{equation}
where $\mathcal{U}[0,T]:=\{u(\cdot)\mid u(\cdot)\in M_{G}^{2}(0,T;\mathbb{R}%
^{m})\}$ and $A(\cdot)$, $C(\cdot)$, $\tilde{B}(\cdot)$, $D(\cdot)$,
$b(\cdot)$, $\sigma(\cdot)$ are deterministic functions. The cost functional
is
\[
J(u(\cdot))=\frac{1}{2}\mathbb{\hat{E}[}\int_{0}^{T}[\langle
Q(t)x(t),x(t)\rangle+2\langle S(t)x(t),u(t)\rangle+\langle
R(t)u(t),u(t)\rangle]dt+\langle Lx(T),x(T)\rangle],
\]
where $Q(\cdot)$, $S(\cdot)$, $R(\cdot)$ are deterministic functions. The
stochastic optimal control problem is to minimize the cost functional over
$\mathcal{U}[0,T]$.

In the following, the variable $t$ will be suppressed. We suppose the
functions satisfy the following conditions:%
\begin{equation}
\left\{
\begin{array}
[c]{l}%
A,C\in L^{\infty}(0,T;\mathbb{R}^{n\times n}),\text{ }\tilde{B}\in L^{\infty
}(0,T;\mathbb{R}^{n\times m}),\text{ }D\in C(0,T;\mathbb{R}^{n\times m}),\\
Q\in L^{\infty}(0,T;\mathbb{S}_{n}),\text{ }S\in L^{\infty}(0,T;\mathbb{R}%
^{m\times n}),\text{ }R\in C(0,T;\mathbb{S}_{m}),\\
b,\sigma\in L^{2}(0,T;\mathbb{R}^{n}),\text{ }L\in\mathbb{S}_{n}%
\end{array}
\right.  \label{SLQcon}%
\end{equation}%
\begin{equation}
R\gg0,\text{ }Q-SR^{-1}S^{T}\geq0,\text{ }L\gg0, \label{Standard}%
\end{equation}
where $R\gg0$ means that there exists a $\delta>0$ such that $R\geq\delta I$
and similarly for $L\gg0$. In this case, the Hamiltonian function is%
\begin{equation}
H(x,u,p,q,t)=\langle p,Ax+\tilde{B}u+b\rangle+\langle q,Cx+Du+\sigma
\rangle\gamma(t)+\frac{1}{2}(\langle Qx,x\rangle+2\langle Sx,u\rangle+\langle
Ru,u\rangle). \label{LQham}%
\end{equation}
Let $\bar{u}$ be an optimal control. By maximum principle which still holds
for this case, there exists a $P^{\ast}\in\mathcal{P}$ such that%
\begin{equation}
\left\{
\begin{array}
[c]{l}%
E_{P^{\ast}}[\bar{K}(T)]=0;\\
\tilde{B}^{T}p(t)+D^{T}q(t)\gamma(t)+S\bar{x}(t)+R\bar{u}(t)=0,\text{ under
}P^{\ast}\text{,}%
\end{array}
\right.  \label{LQmax}%
\end{equation}
where $(p(\cdot),q(\cdot),N(\cdot))$ is the solution of the following adjoint
equation under the probability $P^{\ast}$%
\begin{equation}
\left\{
\begin{array}
[c]{rl}%
-dp(t)= & [Q\bar{x}(t)+S^{T}\bar{u}(t)+A^{T}p(t)+C^{T}q(t)\gamma
(t)]dt-q(t)dB(t)-dN(t),\\
p(T)= & L\bar{x}(T).
\end{array}
\right.  \label{LQadj}%
\end{equation}
Suppose that
\begin{equation}
p(t)=P(t)\bar{x}(t)+\varphi(t) \label{Lqcon1}%
\end{equation}
with $P(\cdot)\in C^{1}([0,T],\mathbb{S}_{n})$, $\varphi(\cdot)\in
C^{1}([0,T],\mathbb{R}^{n})$. Applying It\^{o}'s formula to $p(t)$, we can get%
\[
q(t)=P(t)C(t)\bar{x}(t)+P(t)D(t)\bar{u}(t)+P(t)\sigma(t),
\]%
\[
\dot{P}\bar{x}+PA\bar{x}+P\tilde{B}\bar{u}+Pb+\dot{\varphi}+Q\bar{x}+S^{T}%
\bar{u}+A^{T}p+C^{T}q\gamma=0.
\]
Combining (\ref{LQmax}), (\ref{Lqcon1}) and the above two equalities, we can
obtain that%
\begin{equation}
q=[PC-PD(R+D^{T}PD\gamma)^{-1}(\tilde{B}^{T}P+S+D^{T}PC\gamma)]\bar
{x}-PD(R+D^{T}PD\gamma)^{-1}(\tilde{B}^{T}\varphi+D^{T}P\sigma\gamma)+P\sigma,
\label{Lqcon2}%
\end{equation}%
\begin{equation}
\bar{u}=-(R+D^{T}PD\gamma)^{-1}[(\tilde{B}^{T}P+S+D^{T}PC\gamma)\bar{x}%
+\tilde{B}^{T}\varphi+D^{T}P\sigma\gamma], \label{Lqcon3}%
\end{equation}
and the following Riccati equation for $P$%
\begin{equation}
\left\{
\begin{array}
[c]{l}%
\dot{P}+PA+A^{T}P+C^{T}PC\gamma+Q\\
-(\tilde{B}^{T}P+S+D^{T}PC\gamma)^{T}(R+D^{T}PD\gamma)^{-1}(\tilde{B}%
^{T}P+S+D^{T}PC\gamma)=0,\text{ a.e. }t\in\lbrack0,T],\\
P(T)=L,
\end{array}
\right.  \label{Lqcon4}%
\end{equation}%
\begin{equation}
\left\{
\begin{array}
[c]{l}%
\dot{\varphi}+[A-\tilde{B}(R+D^{T}PD\gamma)^{-1}(\tilde{B}^{T}P+S+D^{T}%
PC\gamma)]^{T}\varphi\\
+[C-D(R+D^{T}PD\gamma)^{-1}(\tilde{B}^{T}P+S+D^{T}PC\gamma)]^{T}P\sigma
\gamma+Pb=0,\text{ a.e. }t\in\lbrack0,T],\\
\varphi(T)=0.
\end{array}
\right.  \label{Lqcon5}%
\end{equation}
It is important to note that $P^{\ast}$ is uniquely determined by the choose
of $\gamma$. We choose $\gamma(t)=\bar{\sigma}^{2}$. It is well known that the
Riccati equation (\ref{Lqcon4}) has a unique solution $P\gg0$, and then
equation (\ref{Lqcon5}) has a unique solution $\varphi$. In this case, the
optimal control%
\begin{equation}
\bar{u}=-(R+D^{T}PD\bar{\sigma}^{2})^{-1}[(\tilde{B}^{T}P+S+D^{T}PC\bar
{\sigma}^{2})\bar{x}+\tilde{B}^{T}\varphi+D^{T}P\sigma\bar{\sigma}^{2}],
\label{optmallq}%
\end{equation}
where%
\begin{equation}
\left\{
\begin{array}
[c]{rl}%
d\bar{x}(t)= & [A(t)\bar{x}(t)+\tilde{B}(t)\bar{u}(t)+b(t)]dt+[C(t)\bar
{x}(t)+D(t)\bar{u}(t)+\sigma(t)]dB(t),\\
x(0)= & x_{0},\;x_{0}\in\mathbb{R}^{n}.
\end{array}
\right.  \label{optstalq}%
\end{equation}

In the following, we prove that the above $\bar{u}$ is the optimal control.

\begin{theorem}
Suppose (\ref{SLQcon}) and (\ref{Standard}) hold. Then $\bar{u}$ defined in
(\ref{optmallq}) and (\ref{optstalq}) is the optimal control, where $P$ and
$\varphi$ are solutions for equations (\ref{Lqcon4}) and (\ref{Lqcon5}) with
$\gamma(t)=\bar{\sigma}^{2}$.
\end{theorem}

\begin{proof}
Let $P^{\ast}\in\mathcal{P}$ be the probability such that $\langle
B\rangle(t)=\bar{\sigma}^{2}t$. It is easy to check that $p$, $q$ defined in
(\ref{Lqcon1}) and (\ref{Lqcon2}), and $N=0$ is the solution of the adjoint
equation (\ref{LQadj}) under the probability $P^{\ast}$. Also, it is easy to
check that the Hamiltonian function $H$ is convex with respect to $x$, $u$
and
\[
H_{u}(\bar{x}(t),\bar{u}(t),p(t),q(t),t)=0,\text{ under }P^{\ast}\text{.}%
\]
By Theorem \ref{the-sufficient}, we only need to verify that $E_{P^{\ast}%
}[\bar{K}(T)]=0$. Let $l$ be the solution of the following ODE:%
\[
\left\{
\begin{array}
[c]{l}%
l^{\prime}+\langle\varphi,b\rangle+\frac{1}{2}\bar{\sigma}^{2}\langle
P\sigma,\sigma\rangle\\
-\frac{1}{2}(\tilde{B}^{T}\varphi+\bar{\sigma}^{2}D^{T}P\sigma)^{T}%
(R+D^{T}PD\bar{\sigma}^{2})^{-1}(\tilde{B}^{T}\varphi+\bar{\sigma}^{2}%
D^{T}P\sigma)=0,\text{ a.e. }t\in\lbrack0,T],\\
l(T)=0.
\end{array}
\right.
\]
Set
\[%
\begin{array}
[c]{l}%
\tilde{Y}(t)=\frac{1}{2}\langle P\bar{x},\bar{x}\rangle+\langle\varphi,\bar
{x}\rangle+l,\\
\tilde{Z}(t)=\langle P\bar{x}+\varphi,C\bar{x}+D\bar{u}+\sigma\rangle,\\
\tilde{K}(t)=\frac{1}{2}\int_{0}^{t}\langle P(C\bar{x}+D\bar{u}+\sigma
),C\bar{x}+D\bar{u}+\sigma\rangle d\langle B\rangle(s)-\int_{0}^{t}G(\langle
P(C\bar{x}+D\bar{u}+\sigma),C\bar{x}+D\bar{u}+\sigma\rangle)ds.
\end{array}
\]
By applying It\^{o}'s formula to $\tilde{Y}$ and some simple calculations, we
can get%
\[
\tilde{Y}(t)=\frac{1}{2}\langle L\bar{x}(T),\bar{x}(T)\rangle+\frac{1}{2}%
\int_{t}^{T}[\langle Q\bar{x},\bar{x}\rangle+2\langle S\bar{x},\bar{u}%
\rangle+\langle R\bar{u},\bar{u}\rangle]ds-\int_{t}^{T}\tilde{Z}%
(s)dB(s)-(\tilde{K}(T)-\tilde{K}(t)),
\]
which implies that $\bar{K}(T)=\tilde{K}(T)$. Note that $\langle P(C\bar
{x}+D\bar{u}+\sigma),C\bar{x}+D\bar{u}+\sigma\rangle\geq0$, then we get%
\[
\bar{K}(T)=\frac{1}{2}\int_{0}^{T}\langle P(C\bar{x}+D\bar{u}+\sigma),C\bar
{x}+D\bar{u}+\sigma\rangle d(\langle B\rangle(s)-\bar{\sigma}^{2}s).
\]
Obviously, $E_{P^{\ast}}[\bar{K}(T)]=0$. Thus $\bar{u}$ is the optimal control.
\end{proof}

\begin{remark}
Using the same method, we can obtain the result for the state equation and
cost functional containing the term $\langle B\rangle$. For the Riccati
equation (\ref{Lqcon4}), we only need $R+D^{T}PD\bar{\sigma}^{2}>0$. This case
will be discussed in our forthcoming paper. Note that $R+D^{T}%
PD\underline{\sigma}^{2}<0$ may be hold, so the LQ problem may be infinite for
some $P\in\mathcal{P}$, but it is finite under $G$-expectation. The reason of
this is the uncertainty of probability measures, which is different from
classical LQ problem.
\end{remark}

In the following, we give an example to point out that the LQ problem with
random coefficients is more difficult and $P^{\ast}$ is not the probability
measure such that $\langle B\rangle(t)=\bar{\sigma}^{2}t$.

\begin{example}
We consider the following $1$-dimensional state equation:%
\[
x(t)=\int_{0}^{t}\sqrt{as-\langle B\rangle(s)}dB(s),
\]
where $a>\bar{\sigma}^{2}$ is a constant. The cost functional is
\[
J(u(\cdot))=\frac{1}{2}\mathbb{\hat{E}[}\int_{0}^{T}(at-\langle B\rangle
(t))|u(t)|^{2}d\langle B\rangle(t)+|x(T)|^{2}].
\]
By applying It\^{o}'s formula to $|x(t)|^{2}$, it is easy to check that
\[
J(u(\cdot))=\frac{1}{2}\mathbb{\hat{E}[}\int_{0}^{T}(at-\langle B\rangle
(t))(|u(t)|^{2}+1)d\langle B\rangle(t)].
\]
Obvious, the optimal control $\bar{u}\equiv0$ and $P^{\ast}\in\mathcal{P}$
satisfies
\[
\mathbb{\hat{E}[}\int_{0}^{T}(at-\langle B\rangle(t))d\langle B\rangle
(t)]=E_{P^{\ast}}[\int_{0}^{T}(at-\langle B\rangle(t))d\langle B\rangle(t)].
\]
By simple calculation, we can obtain $P^{\ast}$ is the probability measure
such that
\[
\langle B\rangle(t)=\int_{0}^{t}(\underline{\sigma}^{2}I_{[0,t^{\ast}%
]}(s)+\bar{\sigma}^{2}I_{(t^{\ast},T]}(s))ds,
\]
where $t^{\ast}=\bar{\sigma}^{2}T(a+\bar{\sigma}^{2}-\underline{\sigma}%
^{2})^{-1}$. It is easy to check that this $\bar{u}$ satisfies the maximum
principle in Theorem \ref{Thm-MP-new}.
\end{example}

\section{Appendix}

The following proposition is about some further estimates for Theorems
\ref{variational derivative-y} and \ref{variational ineq}, which is interest
of itself.

\begin{proposition}
Suppose (H1)-(H3) hold. Then

\begin{description}
\item[(1)] for each $u\in\mathcal{U}[0,T]$, there exists a $P^{u}%
\in\mathcal{P}^{\ast}$ such that
\[%
\begin{array}
[c]{l}%
E_{P^{u}}[\Theta^{u}]=\underset{P\in\mathcal{P}^{\ast}}{\sup}E_{P}[\Theta
^{u}],\\
\underset{\rho\rightarrow0}{\lim}E_{P^{u}}[\frac{-\int_{0}^{T}m(t)dK_{\rho
}^{u}(t)}{\rho}]=0;
\end{array}
\]

\item[(2)] there exists a $P^{\ast}\in\mathcal{P}^{\ast}$ such that%
\[%
\begin{array}
[c]{l}%
\underset{P\in\mathcal{P}^{\ast}}{\sup}\underset{u\in\mathcal{U}[0,T]}{\inf
}E_{P}[\Theta^{u}]=\underset{u\in\mathcal{U}[0,T]}{\inf}E_{P^{\ast}}%
[\Theta^{u}],\\
\underset{u\in\mathcal{U}[0,T]}{\inf}(\underset{\rho\rightarrow0}{\lim
}E_{P^{\ast}}[\frac{-\int_{0}^{T}m(t)dK_{\rho}^{u}(t)}{\rho}])=0.
\end{array}
\]

\end{description}
\end{proposition}

\begin{proof}
(1) Consider
\[%
\begin{array}
[c]{rl}%
y_{\rho}^{u}(t)-\bar{y}(t)= & \phi(x_{\rho}^{u}(T))-\phi(\bar{x}(T))+\int%
_{t}^{T}(f_{\rho}^{u}(s)-f(s))ds-\int_{t}^{T}(z_{\rho}^{u}(s)-\bar
{z}(s))dB(s)\\
& -(K_{\rho}^{u}(T)-K_{\rho}^{u}(t))+(\bar{K}(T)-\bar{K}(t)).
\end{array}
\]
By Theorem \ref{variational derivative-y}, there exists a $P^{u}\in
\mathcal{P}^{\ast}$ such that $E_{P^{u}}[\Theta^{u}]=\sup_{P\in\mathcal{P}%
^{\ast}}E_{P}[\Theta^{u}]$. Note that $\bar{K}\equiv0$ under probability
$P^{u}$. Similar as in the proof of Theorem \ref{variational derivative-y}, we
can derive%
\[
E_{P^{u}}[\frac{\int_{0}^{T}m(s)dK_{\rho}^{u}(s)}{\rho}]=E_{P^{u}}[\Theta
^{u}]-\frac{y_{\rho}^{u}(0)-\bar{y}(0)}{\rho}+o(1).
\]
By Theorem \ref{variational derivative-y},
\[
\underset{\rho\rightarrow0}{\lim}\frac{y_{\rho}^{u}(0)-\bar{y}(0)}{\rho
}=E_{P^{u}}[\Theta^{u}],
\]
which implies that%
\[
\underset{\rho\rightarrow0}{\lim}E_{P^{u}}[\frac{-\int_{0}^{T}m(s)dK_{\rho
}^{u}(s)}{\rho}]=0.
\]

(2) For any $P\in\mathcal{P}^{\ast}$, similar analysis as in (1), we have%
\[
E_{P}[\frac{\int_{0}^{T}m(s)dK_{\rho}^{u}(s)}{\rho}]=E_{P}[\Theta^{u}%
]-\frac{y_{\rho}^{u}(0)-\bar{y}(0)}{\rho}+o(1).
\]
Then,%
\begin{equation}
\underset{\rho\rightarrow0}{\lim}E_{P}[\frac{-\int_{0}^{T}m(s)dK_{\rho}%
^{u}(s)}{\rho}]=E_{P^{u}}[\Theta^{u}]-E_{P}[\Theta^{u}]\geq0.
\label{v-equation-y-4}%
\end{equation}
By Minimax Theorem, we can get%
\[
\underset{u\in\mathcal{U}[0,T]}{\inf}\sup_{P\in\mathcal{P}^{\ast}}E_{P}%
[\Theta^{u}]=\sup_{P\in\mathcal{P}^{\ast}}\underset{u\in\mathcal{U}%
[0,T]}{\inf}E_{P}[\Theta^{u}].
\]
By the proof of Theorem \ref{variational ineq}, we can obtain a $P^{\ast}%
\in\mathcal{P}^{\ast}$ such that
\[
\sup_{P\in\mathcal{P}^{\ast}}\underset{u\in\mathcal{U}[0,T]}{\inf}E_{P}%
[\Theta^{u}]=\underset{u\in\mathcal{U}[0,T]}{\inf}E_{P^{\ast}}[\Theta^{u}].
\]
Note that%
\[%
\begin{array}
[c]{l}%
\underset{u\in\mathcal{U}[0,T]}{\inf}\underset{P\in\mathcal{P}^{\ast}}{\sup
}E_{P}[\Theta^{u}]=\underset{u\in\mathcal{U}[0,T]}{\inf}E_{P^{u}}[\Theta
^{u}],\\
E_{P^{u}}[\Theta^{u}]=\underset{P\in\mathcal{P}^{\ast}}{\sup}E_{P}[\Theta
^{u}]\geq E_{P^{\ast}}[\Theta^{u}].
\end{array}
\]
We deduce that%
\[
\underset{u\in\mathcal{U}[0,T]}{\inf}(E_{P^{u}}[\Theta^{u}]-E_{P^{\ast}%
}[\Theta^{u}])=0.
\]
Taking $P=P^{\ast}$ in (\ref{v-equation-y-4}), it yields that%
\[
\underset{u\in\mathcal{U}[0,T]}{\inf}(\underset{\rho\rightarrow0}{\lim
}E_{P^{\ast}}[\frac{-\int_{0}^{T}m(s)dK_{\rho}^{u}(s)}{\rho}])=0.
\]
This completes the proof.
\end{proof}

\end{document}